\documentclass[10pt, leqno]{amsart}
\usepackage[mathcal]{eucal}

\usepackage{Style}

\title{Linear sections of Grassmannians and resonance of vector bundles}
\author{Marian Aprodu, C\u alin Spiridon}
\address{Marian Aprodu: University of Bucharest, Faculty of Mathematics and Informatics, Academiei Str. 14 Bucharest, Romania \& "Simion Stoilow" Institute of Mathematics of the Romanian Academy, P.O. Box 1-764, Bucharest, Romania}
\email{marian.aprodu@fmi.unibuc.ro \& marian.aprodu@imar.ro}

\address{C\u alin Spiridon: University of Bucharest, Faculty of Mathematics and Informatics, Academiei Str. 14 Bucharest, Romania \& "Simion Stoilow" Institute of Mathematics of the Romanian Academy, P.O. Box 1-764, Bucharest, Romania}
\email{cspiridon@imar.ro}
\date{\today}
\thanks{The authors have been partly funded by the project PNRR-III-C9-2022-I8 "Cohomological Hall algebras of smooth surfaces and applications" - CF 44/14.11.2022.}
\thanks{Marian Aprodu is grateful to Gavril Farkas, Claudiu Raicu and Alex Suciu for numerous enlightening discussions on resonance varieties and Koszul modules, as well as for our previous joint work on the topic. The authors thank the referee for the careful reading and the constructive suggestions.}

\begin{document}

\begin{abstract}
The investigation of resonance varieties of vector bundles was initiated in the papers \cite{aprodu_koszul_2024}, \cite{aprodu_reduced_2024}. 
This work revolves around the question of  whether  a given resonance variety is associated with a vector bundle.
We show the existence of a family of natural morphisms on a stratification of the resonance variety to a suitable family of a Quot scheme and provide some applications in the curve case. The existence of this family of morphisms represents an obstruction to affirmatively answering the main question. 
In addition, we study the resonance of restricted universal rank-two quotient bundles over transversal linear sections of the Grassmann varieties $\gr_2(\mathbb{C}^n)$, with a special attention to low-dimensional Grassmannians. These bundles are among the most natural to consider in this context. The analysis for $\gr_2(\mathbb{C}^6)$ shows that any resonance variety in $\p^5$ consisting of fourteen disjoint lines is the resonance of some Mukai bundle as defined in ~\cite{mukai_curves_1992}.
\end{abstract}

\maketitle


\section{Introduction}

The \emph{resonance locus} $\r(V,K)$ associated to a pair $(V,K)$, where $V$ is a finite-dimensional complex vector space and $K\subseteq \bigwedge^2V$ is a subspace with orthogonal complement $K^\perp \subseteq \bigwedge^2V^\vee$, is a projective variety covered by lines, defined as the image --- via the incidence variety --- of the linear section $\gr_2(V^\vee)\cap \p K^\perp$ of the Grassmannian $\gr_2(V^\vee)$. This notion originates from the study of hyperplane arrangements (see for instance \cite{falk_resonance_1997}), and today it plays a significant role in geometric group theory. Papadima and Suciu \cite{papadima_vanishing_2015} showed that resonance varieties are support loci for some naturally defined finitely-generated graded modules over the polynomial algebra, called \emph{Koszul modules}. As such, they carry a natural scheme structure given by the annihilator, which is sometimes non-reduced \cite{aprodu_reduced_2024}. 

The main source of examples comes from groups, starting from the set of decomposable elements in the kernel $K^\perp$ of a cup-product map, see, for example \cite{papadima_vanishing_2015}, \cite{aprodu_koszul_2022}. If this set is empty, then one can effectively estimate the Chen ranks of the group, accomplishing one of the tasks in geometric group theory, see \cite{aprodu_koszul_2022}. Under more relaxed hypotheses, for example, if the resonance is a disjoint union of (isotropic) projective subspaces and reduced, the Chen ranks are still computable, see for example \cite{aprodu_reduced_2024}. Hence, geometric properties of group resonance loci are reflected in algebraic properties of groups. In the case of K\"ahler groups, the existence of decomposable elements in the kernel of the cup-product map is directly related to the existence of fibrations over curves of genus $\ge 2$, by the topological version of Castelnuovo-de Franchis Theorem, see \cite{catanese_1991}.

Another rich source of examples is provided by vector bundle theory. Given a vector bundle $E$ on a smooth projective variety $X$, we consider the set of pairs of global sections that generate a rank-one subsheaf of $E$. The image of this locus in the projectivization $\p H^0(X,E)$ of the space of global sections is the \emph{resonance locus} $\r(E)$ of the vector bundle $E$. When considered with the reduced structure, this locus is the \emph{resonance variety} of the bundle. The resonance of a vector bundle is a useful tool in studying its geometric properties. For instance, the resonance is trivial if and only if the bundle has no subpencils, and, in rank-two, this property could be the starting point for verifying stability. This condition of having trivial resonance was implicitly used by C. Voisin, for suitable Lazarsfeld-Mukai bundles, in her landmark papers \cite{voisin_greens_2002}, \cite{voisin_greens_2005} on Green's conjecture. Like in the case of groups, if the resonance is trivial, one can estimate the Hilbert series of the associated Koszul module. Other geometric properties, like stability or splitability, can be read off the resonance, see \cite{aprodu_koszul_2024}.

\medskip 

This paper is mainly concerned with the following fundamental question.

\begin{ques} \label{ques:General}
Given a vector space $V$ of dimension $n \ge 4$ and a linear subspace $K \subseteq \bigwedge^2V$, does there exist a smooth projective variety $X$ and a vector bundle $E$ on $X$ such that the resonance of the pair $(V,K)$ coincides with the resonance of the vector bundle $E$?
\end{ques}

Since we omit the more refined scheme structure here, the question is asked in terms of set--theoretic equality. It nonetheless admits a deeper formulation that takes also into account the scheme structure: given a pair $(V,K)$ does there exist a smooth projective variety $X$ and a rank-two vector bundle $E$ on $X$ such that $H^0(X,E) \cong V^\vee$ and the orthogonal $K^\perp\subseteq \bigwedge^2V^\vee$ naturally identifies with the kernel of the determinant map of $E$ via this isomorphism?
It was emphasized in \cite{aprodu_reduced_2024} that resonance varieties of vector bundles enjoy special properties, for example, they are set-theoretically disjoint unions of projectivisation of spaces of global sections of \emph{saturated} sub-line bundles. In this paper we continue the discussion begun in \cite{aprodu_koszul_2024}, \cite{aprodu_reduced_2024} and prove more properties of resonance of vector bundles that provide obstructions to Question \ref{ques:General}, see Theorem \ref{thm:Flattening-Stratification-Square} in Section \ref{sec:Resonance}. 

\medskip

Another basic question, which is a special instance of Question \ref{ques:General}, is the following.

\begin{ques} \label{ques:Generic}
 How \emph{generic} are the resonance loci of vector bundles? That is, can generic linear sections of the Grassmannian $\gr_2(\mathbb{C}^n)$ be associated to vector bundles?
\end{ques}

We will positively answer Question \ref{ques:Generic} for small values of the dimension $n$ in Section \ref{sec:LowDim}. 
In the process, we introduce the notion of \emph{transversality} of a vector bundle and prove several geometric properties of the restriction of the universal quotient bundle over the aforementioned linear section in Theorem~\ref{thm:TransvConseq}. The question remains open whether the results obtained in the low-dimensional case reflect a pathological behavior or point to a more general phenomenon.

\medskip

The outline of the paper is as follows. 
In Section \ref{sec:prelim}, we recall some basic facts related to saturation and introduce two natural degeneracy loci associated to a vector bundle on a smooth projective variety.

In Section \ref{sec:Resonance}, we recall the definition of the resonance loci, focusing on the vector bundle case. We also revisit the relationships with saturated sub-line bundles and expand the discussion initiated in \cite[Proposition 6.1]{aprodu_reduced_2024}. We define a stratification of the resonance locus of a bundle $E$ such that, on each stratum, the map to the locus $\w(E)$ of saturated sub-line bundles is defined by a flat family of quotients of the dual bundle $E^\vee$ and hence it gives a morphism to a Quot scheme, Theorem \ref{thm:Flattening-Stratification-Square}. As a consequence, we note that the resonance of a vector bundle is linear if and only if $\w(E)$ is finite, Proposition \ref{prop:LinearResonance}, see also \cite[Proposition 6.1]{aprodu_reduced_2024}.

In Section \ref{sec:Curves} we consider the case of rank-two vector bundles on curves. As in \cite{aprodu_reduced_2024}, the loci $W^1_d(E)$ of $g^1_d$'s that can be embedded into the given rank-two bundle $E$ play a major role. These loci relate to the resonance via a morphism from $\w(E)$ to the disjoint union of $W^1_d(E)$ which yields to a clearer description of the resonance. For ample rank-two bundles on the projective line, we show that the resonance variety is irreducible, even though, besides the case of $\cO_{\p^1}(1) \oplus \cO_{\p^1}(1)$, the corresponding linear section of the Grassmannian is not.

Section \ref{sec:Transversal} focuses on the case where the resonance of a vector bundle arises from a transversal section of the Grassmannian; in this setting, we say that the bundle is \emph{transversal}. Theorem \ref{thm:TransvConseq} is one of the main results of the paper. We show that, for transversal bundles, the resonance is smooth in general, and is very closely related to the incidence variety --- which is also smooth, being a $\p^1$-bundle over the given linear section of $\gr_2(\mathbb{C}^n)$.  
In the case of the projective line, the only ample transversal rank-two bundle is $\o_{\p^1}(1)\oplus \o_{\p^1}(1)$, see Proposition \ref{prop:p1transv}. 

In Section \ref{sec:Universal}, we study the restrictions of the universal rank-two quotient bundles to transversal linear sections of Grassmann varieties $\gr_2(\mathbb{C}^n)$ and determine their resonance loci. Since resonance varieties themselves are constructed starting from linear sections of Grassmannians, these bundles are the most natural to consider. In this case, we show that the resonance arises from the orthogonal linear section, see Theorem \ref{thm_E} and Corollary \ref{cor_E}. 

In Section \ref{sec:LowDim}, we analyze in detail the case of low-dimensional Grassmannians $\gr_2(\mathbb{C}^4)$, $\gr_2(\mathbb{C}^5)$, and $\gr_2(\mathbb{C}^6)$. In each of these cases we provide a complete description of the \mbox{resonance} of the restricted universal rank-two quotient bundle. Moreover, we observe that a linear section $\gr_2(\mathbb{C}^n)\cap \p K$ is transversal if and only if $\gr_2 ((\mathbb{C}^n)^\vee)\cap \p K^\perp$ is transversal, Theorems \ref{thm:LowDim_transversal_n=4,5} and \ref{thm:g(8,15)}. For $n = 4$ or $n = 5$ this follows rather easily from the projective self-duality of the Grassmannian (see, for example \cite[Proposition 2.24]{debarre_gushel-mukai_2018}), however, for $n = 6$ proof is much more involved and uses the foundational work of Mukai \cite{mukai_curves_1992}. Using the properties of the resonance, we prove that, for one-dimensional sections of $\gr_2(\mathbb{C}^6)$, the unique stable bundle with $6$ independent global sections and canonical determinant coincides with the restriction of the universal rank-two quotient bundle, Proposition \ref{prop:E=E_M}. In all these cases of transversal linear sections of $\gr_2(\mathbb{C}^n)$ with $n\in\{4,5,6\}$, Question \ref{ques:Generic} has a positive answer. For instance, our analysis proves that any resonance variety which is the disjoint union of fourteen lines in $\p^5$ is the resonance of a Mukai bundle on a general genus $8$ curve.

\medskip

Throughout the paper, we work over the field of complex numbers, and use the following notation, convention and terminology.

\begin{itemize}
    \item[-] Following \cite{mukai_curves_1992}, a $g^r_d$ on a smooth projective curve $C$ is a line bundle on $C$ of degree $d$ having at least $r+1$ independent global sections (note the difference from the standard terminology). Similarly, a \emph{pencil} on a smooth projective variety is a line bundle with at least two linearly independent global sections.
    \item[-] A \emph{sub-line bundle} of a vector bundle on a smooth projective variety is a rank-one locally free subsheaf. A \emph{subpencil} of a vector bundle is a sub-line bundle which is a pencil.
\end{itemize}

\section{Preliminaries}
\label{sec:prelim}

In this section, we consider a fixed smooth, connected, complex projective variety $X$ and a vector bundle $E$ on $X$. We also denote by $V^\vee$ the space $H^0(X,E)$ of global sections of $E$.

\subsection{Saturation of subsheaves}
The saturation of subsheaves of the locally free sheaf $E$ is a well-known classical method of "completing" subsheaves in a certain sense. We recall the definitions and some of the fundamental properties.

\begin{defn}
A subsheaf $\cF \subseteq E$ is called \emph{saturated} if the quotient $E/\cF$ is torsion-free.
\end{defn}

Note that any saturated subsheaf of $E$ is reflexive, see for example \cite[p. 28]{Huybrechts-Lehn_2010}. In particular, a rank-one saturated subsheaf of $E$ is invertible.

The \emph{saturation} of a subsheaf $\cF$ of $E$ (see \cite[Definition 1.1.5]{Huybrechts-Lehn_2010}) is defined by

\[
\cF^{\mathrm{sat}}:=\ker(E \longrightarrow (E/\cF)/(\mathrm{tors}(E/\cF)));
\]
It is immediate that
\[
\cF^{\mathrm{sat}}\cong (\cF^{\vee\vee})^{\mathrm{sat}}=\ker(E \longrightarrow (E/\cF^{\vee\vee})/(\mathrm{tors}(E/\cF^{\vee\vee}))).
\]

By the definition, it is straightforward to check the following.

\begin{prop} \label{prop:sat}
Let $\cF \subseteq \cG \subseteq E$ be two subsheaves of $E$. Then $\cF^{\text{sat}} \subseteq \cG^\text{sat}$. Moreover, if $\cF$ and $\cG$ have the same rank, then $\cF^{\text{sat}} = \cG^\text{sat}$.
\end{prop}

An important source of examples of subsheaves is provided by evaluation maps. Specifically, if $U\subseteq V^\vee$ is a non-zero subspace, consider 
\[
\operatorname{ev}_U : U \otimes \o_X \longrightarrow E
\]
the evaluation on sections in $U$. Then $\im(\operatorname{ev}_U)$ is a subsheaf of $E$ that is not saturated in general. If its rank equals one, then its saturation is an invertible sheaf on $X$, which we denote by $L_U$. If $U$ is one-dimensional, generated by a section $a$, then we use the notation $L_a$.

Using the saturations of type $L_a$, it readily follows from Proposition \ref{prop:sat} that

\begin{prop} \label{prop:sat_section}
If $\cL_1$ and $\cL_2$ are two rank-one subsheaves of $E$ having a common non-zero section, then $\cL_1^\text{sat} = \cL_2^\text{sat}$.
\end{prop}

Saturation will turn out to be essential in the constructions that follow in the next section.

\subsection{Two degeneracy loci associated to a vector bundle}
In the subsequent sections, we will work with two degeneracy loci that we describe here. The first one lives in the product $X\times \p V^\vee$ and gives a natural scheme structure on the incidence locus
\begin{equation}
\label{eqn:Incidence_XxP}
\{(x,[a]):\ x\in V(a)\}.
\end{equation}
The construction goes as follows. We consider the projections $\mu_1 : X \times \p V^\vee \longrightarrow X$ and $\mu_2 : X \times \p V^\vee \longrightarrow \p V^\vee$. We have a natural identification 
\[
H^0(X\times \p V^\vee,\mu_1^*(E)\otimes \mu_2^*(\o_{\p V^\vee}(1)))\cong V^\vee\otimes V,
\]
and the section corresponding to $\mathrm{id}_{V^\vee}$ induces a sheaf morphism
\begin{equation}
\label{eqn:Incidence-scheme_XxP}
\mu_1^*(E^\vee )\otimes \mu_2^*(\o_{\p V^\vee}(-1)) \longrightarrow \o_{X\times \p V^\vee}
\end{equation}
whose cokernel is supported on the incidence locus (\ref{eqn:Incidence_XxP}). Denote by $V(E)$ the scheme defined by the morphism (\ref{eqn:Incidence-scheme_XxP}), so that we have an exact sequence of sheaves
\begin{equation}
\label{eqn:Incidence-exact-seq_XxP}
\mu_1^*(E^\vee )\otimes \mu_2^*(\o_{\p V^\vee}(-1))\longrightarrow \o_{X\times \p V^\vee}\longrightarrow \o_{V(E)}\longrightarrow 0.
\end{equation}
Obviously, the ideal sheaf $\mathcal{I}_{V(E)}$ coincides with the image of the morphism (\ref{eqn:Incidence-scheme_XxP}). Also, for any non-zero section $a\in V^\vee$, the restriction of the sequence (\ref{eqn:Incidence-exact-seq_XxP}) to $X_a := X \times \{[a]\}$ 
gives an exact sequence
\begin{equation}
\label{eqn:exact-seq-a}
E^\vee\longrightarrow \o_X\longrightarrow \o_{V(a)}\longrightarrow 0,
\end{equation}
where $V(a)\subseteq X$ is the zero scheme of the section $a$.

The second degeneracy locus that we will be using lies inside the product $X\times \gr_2(V^\vee)$. Denote by $\pi_1:X\times \gr_2(V^\vee) \longrightarrow X$ and $\pi_2:X\times \gr_2(V^\vee) \longrightarrow \gr_2(V^\vee)$ the two canonical projections. Denote by $\mathcal{U}^\vee$ the universal rank-two quotient bundle on $\gr_2(V^\vee)$, whose space of global sections naturally identifies with $V$. Hence, we have a natural isomorphism
\[
H^0(X\times \gr_2(V^\vee),\pi_1^*(E)\otimes \pi_2^*(\mathcal{U}^\vee))\cong V^\vee\otimes V,
\]
and the section corresponding to $\mathrm{id}_{V^\vee}$ induces a sheaf morphism
\begin{equation}
\label{eqn:Degeneration-_XxG}\pi_1^*(E^\vee ) \longrightarrow \pi_2^*(\mathcal{U}^\vee)
\end{equation}
whose cokernel we denote by $\mathcal{C}$, so that we have an exact sequence
\begin{equation}
\label{eqn:Degeneration-exact-seq_XxG}
\pi_1^*(E^\vee )\longrightarrow \pi_2^*(\mathcal{U}^\vee)\longrightarrow \mathcal{C}\longrightarrow 0.
\end{equation}
For any $\Lambda \in \gr_2( V^\vee)$, the restriction to $X_\Lambda := \pi_2^{-1}(\Lambda)$ of the exact sequence (\ref{eqn:Degeneration-exact-seq_XxG}) is
\begin{equation}
\label{eqn:exact-seq-Lambda}
E^\vee \longrightarrow \Lambda^\vee \otimes \cO_X \longrightarrow \mathcal{C}|_{X_\Lambda} \longrightarrow 0.
\end{equation}

In what follows, we will explicitly use the sheaves $\cO_{V(E)}$ and $\mathcal{C}$ and the flattening stratifications they define in connection with the geometry of the resonance.

\section{Resonance of vector bundles}
\label{sec:Resonance}

\subsection{Generalities}
We begin by recalling the definition of the projective version of the resonance variety introduced by Papadima and Suciu in \cite{papadima_vanishing_2015}, see also \cite{aprodu_koszul_2024} and \cite{aprodu_reduced_2024}. Let $V$ be a complex vector space of dimension $n\ge 4$ and $K \subseteq \bigwedge^2V$ a linear subspace. Denote by $K^\perp\subseteq \bigwedge^2V^\vee$ the orthogonal of $K$. The \textit{projectivized resonance} of the pair $(V,K)$ is the following subvariety of $\p V^\vee$:
\[
\r(V,K) := \left\{[a]\in \p V^\vee : \exists \ b\in V^\vee \ \text{such that}\ 0\neq a \wedge b \in K^\perp \right\}.
\]
By definition, $\r(V,K)$ is the variety swept out by the projective lines corresponding to the points of $\g \cap \p K^\perp$, where $\g \hooklongrightarrow \p(\bigwedge^2V^\vee)$ is the Grassmannian of lines in $\p V^\vee$, or two-dimensional subspaces in $V^\vee$, in its Pl\"ucker embedding. More precisely, if we consider the incidence diagram
\begin{align} \label{diagrama_incidenta}
\begin{tikzcd}[ampersand replacement = \&, column sep = small]
    \Xi_V \arrow[r, "\text{pr}_2\ "] \arrow[d, swap, "\text{pr}_1"] \& \g \& \\
    \p V^\vee \& 
\end{tikzcd}
\end{align}
where $\Xi_V = \{([a], \Lambda): a \in \Lambda\}\subseteq \p V^\vee \times \g$, then $\r(V,K) = \text{pr}_1(\text{pr}_2^{-1}(\g \cap \p K^\perp))$. Here, $\text{pr}_2$ realizes $\Xi_V$ as the projectivization of the universal rank-two subbundle on $\g$, while $\text{pr}_1$ identifies $\Xi_V$ with the projectivization of the tangent bundle on $\p V^\vee$. In particular, both $\text{pr}_1$ and $\text{pr}_2$ are smooth morphisms.
When endowed with the reduced structure, we refer to $\r(V,K)$ as the \emph{resonance variety}. Note, however, that other natural scheme structures can also be defined on $\r(V,K)$, see \cite{aprodu_reduced_2024}. These scheme structures are related to the fact that the resonance locus is the set-theoretic support of the \emph{Koszul module} $W(V,K)$ which is the graded $\mathrm{Sym}(V)$--module defined as the homology at the middle of the complex
\[
K\otimes \mathrm{Sym}(V)\longrightarrow V\otimes \mathrm{Sym}(V)\longrightarrow \mathrm{Sym}(V)
\]
induced by the classical Koszul exact complex
\[
\textstyle{\bigwedge^2} V\otimes \mathrm{Sym}(V)\longrightarrow V\otimes \mathrm{Sym}(V)\longrightarrow \mathrm{Sym}(V).
\]
For more details on Koszul modules and their relations with resonance, we refer to \cite{papadima_vanishing_2015}, \cite{aprodu_koszul_2022}, \cite{aprodu_reduced_2024}.

\medskip

Under additional extra hypotheses --- for example in the finite case --- the intersection $\g \cap \p K^\perp$ is the \emph{Fano variety} of the resonance, i.e. variety of projective lines entirely contained in $\r(V,K)$. However, in general, the intersection is only included in the Fano variety, without being equal to it. For example, there exist non-trivial cases where the resonance is the entire projective space, and hence the corresponding Fano variety coincides with the Grassmannian. Another simple instance will be provided by Proposition \ref{prop_n = 4, k = 3}. 

\medskip

This general construction has been related to vector bundles in \cite{aprodu_koszul_2024} and \cite{aprodu_reduced_2024} in the following natural way. Let $X$ be a smooth projective variety and $E$ a vector bundle on $X$ of rank at least two. Consider the second determinant map
\[
d_2: \textstyle{\bigwedge^2}H^0(X,E) \longrightarrow H^0(X,\textstyle{\bigwedge^2E}).
\]
In the rank-two case, it is common to denote $d_2$ by $\det$. 

Taking $V := H^0(X,E)^\vee$ and $K := \ker(d_2)^\perp$, yields the \textit{resonance variety} $\r(E): = \r(V,K)$ associated with the vector bundle $E$.  We also use the notation $\Xi_E$ for the incidence variety.

One typical example is obtained when $E=\Omega^1_X$. In this case, the map in question is the natural wedge product map on forms 
\[
\textstyle{\bigwedge^2}H^0(\Omega^1_X) \longrightarrow H^0(\Omega^2_X),
\]
and the resonant elements are precisely the pullbacks of forms on curves of genus $\ge 2$ via suitable fibrations on $X$, see the Castelnuovo-de Franchis Theorem, for example~\cite{catanese_1991}.

\medskip

In the general case of an arbitrary vector bundle $E$, we denote the intersection $\g\cap \p K^\perp$ by $\g(E)$.
There are two possible scheme structures that can be defined on $\mathbb{G}(E)$: the reduced structure, on one hand, and the scheme-theoretic intersection structure, on the other. Unless otherwise stated, we shall consider the reduced structure on $\mathbb{G}(E)$ in what follows. 

By  definition, it is readily seen that $\g(E)$ consists of all two-dimensional subspaces $\Lambda\subseteq H^0(X,E)$ that generate a rank-one subsheaf of $E$. In order to obtain more specific geometric descriptions of $\g(E)$ and $\r(E)$ we shall use \emph{saturated sub-line bundles}.

We record the following useful property of saturations that follows directly from Proposition \ref{prop:sat_section}, see also \cite[Proposition 6.1]{aprodu_reduced_2024}.

\begin{prop}
\label{prop_gE} 
Let $\Lambda_1, \Lambda_2 \in \g(E)$ such that $\Lambda_1 \cap \Lambda_2 \neq \{0\}$. If $\cL_1$ is the rank-one subsheaf generated by $\Lambda_1$ and $\cL_2$ is the rank-one subsheaf generated by $\Lambda_2$, then $\cL_1^{\operatorname{sat}} = \cL_2^{\operatorname{sat}}$. 
\end{prop}

In the absence of the saturation condition, the equality $\cL_1 = \cL_2$ is not guaranteed, as the example below shows.

\begin{exmp}
\label{ex:O(2)O(2)}
Let $E=\o_{\p^1}(2)\oplus\o_{\p^1}(2)$, and consider two sub-line bundles, $\cL_1$ and $\cL_2$, isomorphic to $\o_{\p^1}(1)$, but embedded in $E$ in two different ways. The first embedding is given by the map $(x_0,x_0)$, while the second is defined by the map $(x_1,x_1)$, where $x_0,x_1$ are the homogeneous coordinates on the projective line. The points in the Grassmannian \mbox{correspond} to the spaces generated by the global sections $\{(x_0^2,x_0^2), (x_0x_1,x_0x_1)\}\subseteq H^0(E)$ and $\{(x_0x_1,x_0x_1), (x_1^2,x_1^2)\}\subseteq H^0(E)$, respectively. Note that the two spaces have a common non-zero section, but obviously $\cL_1 \neq \cL_2$. However, their saturation is the same, namely $\o_{\p^1}(2)$ embedded in $E$ by the map $(1,1)$. 
\end{exmp}

Another consequence of Proposition \ref{prop:sat_section} is that the sub-line bundle $L_\Lambda$ introduced in the previous section is the only saturated sub-line bundle $L$ of $E$ such that $\Lambda \subseteq H^0(X,L)$.
In particular, the set $\g(E)$ can be identified with the set of pairs $(L,\Lambda)$, where $L\subseteq E$ is a saturated sub-line bundle of $E$ and $\Lambda\subseteq H^0(X,L)$ is a two-dimensional subspace. Note that if $h^0(X, L_\Lambda)\ge 3$, then $L_\Lambda=L_{\Lambda'}$ for any $\Lambda'\in\gr_2(H^0(X, L_\Lambda))$.

\medskip

The advantage of working with saturated sub-line bundles is that it provides a natural well-defined map from the resonance to the set of saturated subpencils of the vector bundle in question, see \cite[Proposition 6.1]{aprodu_reduced_2024}. Note that we do have a map from $\g(E)$ to the set of all subpencils of $E$, sending a subspace $\Lambda$ to the double dual of the sheaf it generates. However, Example \ref{ex:O(2)O(2)} shows that this map does not induce a well-defined map on the resonance. Hence, saturation becomes crucial. In the next subsection, we will exploit this observation and expand the discussion initiated in  \cite{aprodu_koszul_2024}, \cite{aprodu_reduced_2024} concerning the geometric properties of the projectivized resonance. Particularly, we have a commutative diagram (\ref{diagrama_rezonanta_multimi}) that will be decomposed into commutative diagrams of algebraic varieties after a suitable stratification.

\medskip

The discussion above proves that the resonance variety $\r(E)$ is the set-theoretic disjoint union of the projective spaces $\p H^0(X,L)$, where $L\subseteq E$ is a saturated sub-line bundle of $E$ having at least two independent global sections.
This property distinguishes the resonance of vector bundles from other resonance loci which may be a non-disjoint union of (maximal) projective subspaces. For further discussion, we refer to the next subsection.

\subsection{Resonance, Quot schemes and flattening stratifications}
Throughout this section, if $(X, \o_X(1))$ is a polarized projective variety and $\cF$ is a coherent sheaf on $X$, we denote by $P_\cF$ the Hilbert polynomial of $\cF$ and by $P_X$ the Hilbert polynomial of $X$.

As mentioned above, $\g(E)$ can be identified with the set of pairs $(L,\Lambda)$ with $L\subseteq E$ saturated and $\Lambda\subseteq H^0(X,L)$ of dimension two. Note that forgetting $\Lambda$ yields a map to the locus of saturated subpencils of $E$. 
In this subsection, we endow this locus with a natural structure of a projective scheme.

\begin{defn-lem} \label{defn-lem_w(E)}
Let $\w(E)$ be the set of all saturated subpencils $L\subseteq E$. Then $\w(E)$ naturally identifies with a subset of the Quot scheme $\operatorname{Quot}_{E^\vee/X}$.
\end{defn-lem}

\begin{proof}
Consider the map
\begin{align*}
\sigma: \{\text{saturated sub-line bundles of $E$}\} \longrightarrow \operatorname{Quot}_{E^\vee/X}, \ \ \  L \longmapsto L^\sigma:=\mathrm{Im}(E^\vee\to L^\vee)
\end{align*}
and notice that $\sigma(L)^\vee = L$ and hence this map is injective.
\end{proof}

Fix a polarization $H = \cO_X(1)$ on $X$ and recall that, scheme-theoretically
\[
\operatorname{Quot}_{E^\vee/X} = \textstyle{\coprod_\Phi}\operatorname{Quot}^{\Phi,H}_{E^\vee/X}
\]
where $\operatorname{Quot}^{\Phi,H}_{E^\vee/X}$ is the Quot scheme of quotients of $E^\vee$ with Hilbert polynomial $\Phi$. Now put
\[
\w^\Phi(E) := \w(E)\cap \operatorname{Quot}^{\Phi,H}_{E^\vee/X}
\]
and note that $\w^\Phi(E)$ is locally closed in the corresponding Quot scheme, since being saturated is an open property. Unless otherwise stated, $\w^\Phi(E)$ will be considered as being endowed with the reduced structure, and hence we have an identification of algebraic varieties
\begin{equation}
    \label{eqn:W(E)_strata}
    \w(E) = \textstyle{\coprod}_\Phi \w^\Phi(E).
\end{equation}

Denoting by $\Xi(E)$ the inverse image of $\g(E)$ in the incidence variety $\Xi_E$ and taking $p_1$ and $p_2$ the corresponding projections to $\r(E)$ and $\g(E)$ respectively, it is immediate that we have a commutative diagram of \emph{sets}

\begin{align} \label{diagrama_rezonanta_multimi}
\begin{tikzcd}[ampersand replacement = \&, column sep = small]
        \Xi(E) \arrow[rr, "p_2"] \arrow[d, swap, "p_1"] \& {}  \& \g(E) \arrow[d, "\gamma"] \\
        \r(E) \arrow[rr, "\rho"]  \& {}  \& \w(E) \mathrlap{{} \hooklongrightarrow \operatorname{Quot}_{X,E^\vee}}
\end{tikzcd}
\end{align}
where $\gamma(\Lambda)=L_\Lambda$ and $\rho$ sends any resonant element $[a]$ to the saturation $L_a$ of the subsheaf of $E$ generated by $a$.  
Note that, taking the fibers of $\gamma$ and $\rho$, we obtain a diagram similar to (\ref{diagrama_incidenta})
\begin{align} \label{diagrama_fiber_multimi}
\begin{tikzcd}[ampersand replacement = \&]
    \Xi_L \arrow[r] \arrow[d] \& \gr_2(H^0(X,L))=\gamma^{-1}(L) \arrow[d] \\
    \p H^0(X,L)=\rho^{-1}(L) \arrow[r] \& \{L\}
\end{tikzcd}
\end{align}
where $\Xi_L$ is the incidence variety inside $\p H^0(X,L)\times \gr_2(H^0(X,L))$. As mentioned earlier, the fiber of the morphism $\Xi_L \longrightarrow \gr_2(H^0(X,L))$ over some $\Lambda$ is $\p \Lambda$ and the fiber of the morphism $\Xi_L \longrightarrow \p H^0(X,L)$ over a point $[a]$ coincides with $\p (H^0(X,L)/\ell_a)$, where $\ell_a \subseteq H^0(X,L)$ is the line generated by $a$. 

\medskip

We investigate next the regularity properties of $\rho$ and $\gamma$. To this end, we shall use the restrictions of the sheaves $\o_{V(E)}$ and $\mathcal{C}$ introduced in the previous section, see (\ref{eqn:Incidence-exact-seq_XxP}) and (\ref{eqn:Degeneration-exact-seq_XxG}), to $X \times \r(E)$ and $X \times \g(E)$ respectively, and their associated flattening stratifications. The flattening stratification was introduced by A. Grothendieck and simplified by D. Mumford with the aim of working with arbitrary families of sheaves, and the definition is the following, see Nitsure's survey \cite[Chapter~5]{fantechi_etal2005}.

\begin{thm} \label{thm:flattening_stratification}
Let $T$ be a scheme of finite type over $\mathbb{C}$ and $\cF$ be a coherent sheaf on $X\times T$. Then the set $I$ of Hilbert polynomials of the restrictions of $\cF$ to the fibers of the projection  $X \times T \rightarrow T$ is a finite set. Moreover, for each $\varphi \in I$, there exists a locally closed subscheme $T^\varphi$ of $T$, such that the underlying set of $T^\varphi$ is formed by those $t$ such that the Hilbert polynomial of $\cF|_{X\times \{t\}}$ is~$\varphi$. In particular, 
\[
T = \textstyle{\coprod}_{\varphi\in I} T^\varphi,
\]
set-theoretically. Moreover, the finite set $I$ can be totally ordered by $\varphi < \psi$ if $\varphi(m) < \psi(m)$ for $m \gg 0$ and 
\[
\overline{T^\varphi} \subseteq \textstyle{\bigcup}_{\varphi \le \psi} T^\psi,
\]
or, equivalently, $\overline{T^\varphi}\setminus T^\varphi \subseteq \textstyle{\bigcup}_{\varphi < \psi} T^\psi$.
In particular, if $\psi$ is the maximal Hilbert polynomial in the set $I$, then the corresponding stratum $T^{\psi}$ is closed.
\end{thm}

We are now able to prove the following.

\begin{thm} \label{thm:Flattening-Stratification-Square}
    Notation as above. For any Hilbert polynomial $\Phi$, the subsets $\rho^{-1}(\w^\Phi(E))\subset \r(E)$ and $\gamma^{-1}(\w^\Phi(E))\subset\g(E)$ are locally closed, and the restrictions 
    \[
    \rho^\Phi:\rho^{-1}(\w^\Phi(E)) \longrightarrow \w^\Phi(E)
    \]
    of $\rho$ and 
    \[
    \gamma^\Phi:\gamma^{-1}(\w^\Phi(E)) \longrightarrow \w^\Phi(E)
    \]
    of $\gamma$, respectively, are morphisms of algebraic varieties. In particular, we have a commutative diagram of algebraic varieties:
\begin{center} \label{diagrama_rezonanta_thm}
\begin{tikzcd}[ampersand replacement = \&, column sep = small]
        \Xi^\Phi(E) \arrow[rr, "p_2"] \arrow[d, swap, "p_1"] \& {}  \& \gamma^{-1}(\w^\Phi(E)) \arrow[d, "\gamma^\Phi"] \\
        \rho^{-1}(\w^\Phi(E)) \arrow[rr, "\rho^\Phi"]  \& {}  \& \w^\Phi(E) \mathrlap{{} \hooklongrightarrow \operatorname{Quot}^{\Phi,H}_{E^\vee/X}}
\end{tikzcd}
\end{center}
where $\Xi^\Phi(E)$ denotes the incidence variety between $\rho^{-1}(\w^\Phi(E))$ and $\gamma^{-1}(\w^\Phi(E))$.
Furthermore, we have the following inclusions
\begin{equation} \label{eqn:Strata_Reverse}
\overline{\rho^{-1}(\w^\Phi(E))} \subseteq \textstyle{\bigcup}_{\Psi \le \Phi}{\rho^{-1}(\w^\Psi(E))}
\mbox{ and }
\overline{\gamma^{-1}(\w^\Phi(E))} \subseteq \textstyle{\bigcup}_{\Psi \le \Phi}{\gamma^{-1}(\w^\Psi(E))}.
\end{equation}
\end{thm}

\begin{proof}
The idea is to prove that the inverse images in question coincide with the strata induced by the flattening stratifications of some suitable coherent sheaves, and apply Theorem \ref{thm:flattening_stratification}. Moreover, the strata will be indexed in reverse order with respect to the set of Hilbert polynomials $\{\Phi\}$ (i.e. a larger Hilbert polynomial $\Phi$ will correspond to more generic quotients), which explains the inclusions (\ref{eqn:Strata_Reverse}).
The restrictions of $\o_{V(E)}$ to $X\times\r(E)$ and $\mathcal{C}$ to $X\times\g(E)$ provide a solution to the problem, as shown below.

We begin by restricting the sequence (\ref{eqn:Incidence-exact-seq_XxP}) to $X \times \r(E)$ to obtain the exact sequence
\[
\mu_1^*(E^\vee) \otimes \mu_2^*(\o_\p(-1))|_{X \times \r(E)} \longrightarrow \o_{X \times \r(E)} \longrightarrow \o_{V(E) \cap (X \times \r(E))} \longrightarrow 0
\]
where ${V(E) \cap (X \times \r(E))}$ is the scheme-theoretic intersection. 
Applying Theorem \ref{thm:flattening_stratification} for the scheme $T = \r(E)$ and the coherent sheaf $\cF = \o_{V(E) \cap (X \times \r(E))}$ on $X \times \r(E)$, we obtain the strata $\{\r^\varphi(E)\}_\varphi$ on the resonance $\r(E)$, indexed by Hilbert polynomials $\varphi$. By the flatness of  $\o_{V(E) \cap (X \times \r^\varphi(E))}$ and of $\o_{X \times \r^\varphi(E)}$ over $\r^\varphi(E)$ we infer that $\mathcal{I}_{V(E) \cap (X \times \r^\varphi(E))}$ is flat over $\r^\varphi(E)$. Using the exact sequence (\ref{eqn:exact-seq-a}), it follows that, for any $[a] \in \r^\varphi(E)$, the restriction of this ideal sheaf to $X_a = X \times \{[a]\}$ is precisely $L_a^\sigma$.
Therefore, since $P_{L^\sigma_a} = P_X - P_{\cO_{V(a)}}$, we have the identification of strata
\[
\r^{P_X - \Phi}(E) = \rho^{-1}(\w^\Phi(E))
\]
and the restriction of the map $\rho$ gives a morphism of algebraic varieties
\begin{equation} \label{eq:rho_Phi}
\rho^\Phi: \r^{P_X - \Phi}(E) \longrightarrow \w^\Phi(E).
\end{equation}

We now restrict the sequence (\ref{eqn:Degeneration-exact-seq_XxG}) to $X \times \g(E)$ to get the exact sequence
\[
\pi_1^*E^\vee|_{X \times \g(E)} \longrightarrow \pi_2^* \cU^\vee|_{X \times \g(E)} \longrightarrow \mathcal{C}|_{X \times \g(E)} \longrightarrow 0.
\]

We apply Theorem \ref{thm:flattening_stratification} for the scheme $T = \g(E)$ and the coherent sheaf $\mathcal{C}|_{X \times \g(E)}$ on $X \times \g(E)$ and obtain the strata $\{\g^\psi(E)\}_\psi$ indexed by the set of Hilbert polynomials $\psi$. Let 
\[
\cK_\psi = \ker\left(\pi_2^* \cU^\vee|_{X \times \g^\psi(E)} \longrightarrow \mathcal{C}|_{X \times \g^\psi(E)}\right)
\]
By the definition of $\g(E)$ and the flatness of $\mathcal{C}|_{X \times \g^\psi(E)}$ we infer that $\cK_\psi$ is flat over $\g^\psi(E)$ and $\cK_\psi|_{X_\Lambda} = L^\sigma_\Lambda$, for any $\Lambda \in \g^\psi(E)$, where $X_\Lambda = X \times \{\Lambda\}$. By the exact sequence
\[
0 \longrightarrow L^\sigma_\Lambda \longrightarrow \Lambda^\vee \otimes \o_X \longrightarrow \mathcal{C}|_{X_\Lambda} \longrightarrow 0
\]
we obtain the following identity between Hilbert polynomials $P_{L^\sigma_\Lambda} = 2 P_X - P_{\mathcal{C}|_{X_\Lambda} }$ and hence 
\[
\g^{2P_X - \Phi}(E) = \gamma^{-1}(\w^\Phi(E))
\]
and the restriction of the map $\gamma$ gives a morphism of algebraic varieties:
\begin{equation} \label{eq:gamma_Phi}
\gamma^\Phi: \g^{2P_X - \Phi}(E) \longrightarrow \w^\Phi(E).
\end{equation}

Consequently, the Hilbert polynomial stratification of the incidence diagram between $\r(E)$ and $\g(E)$ can be completed to a commutative diagram of algebraic varieties (\ref{diagrama_rezonanta}), which has a fibered version (\ref{diagrama_fiber}) over a given point $L \in \w(E)$:

\noindent
\begin{minipage}{0.59\textwidth}
\begin{align} \label{diagrama_rezonanta}
\begin{tikzcd}[ampersand replacement = \&, column sep = small]
        \Xi^\Phi(E) \arrow[rr, "p_2"] \arrow[d, swap, "p_1"] \& {}  \& \g^{2 P_X - \Phi}(E) \arrow[d, "\gamma^\Phi"] \\
        \r^{P_X - \Phi}(E) \arrow[rr, "\rho^\Phi"]  \& {}  \& \w^\Phi(E) \mathrlap{{} \hooklongrightarrow \operatorname{Quot}^{\Phi,H}_{E^\vee/X}}
\end{tikzcd}
\hphantom{{} \hooklongrightarrow \operatorname{Quot}_{X,E^\vee}}
\end{align}
\end{minipage}
\begin{minipage}{0.4\textwidth}
\begin{align} \label{diagrama_fiber}
\begin{tikzcd}[ampersand replacement = \&]
    \Xi_L \arrow[r] \arrow[d] \& \gr_2(H^0(L)) \arrow[d] \\
    \p H^0(L) \arrow[r] \& \{L\}
\end{tikzcd}
\end{align}
\end{minipage}

The relations between the strata (\ref{eqn:Strata_Reverse}) follow from Theorem \ref{thm:flattening_stratification}, since $\Phi < \Psi$ if and only if $m P_X - \Phi > m P_X - \Psi$, for $m \in \{1,2\}$.
\end{proof}

We will illustrate in Section \ref{sec:Curves} how the existence of these flattening stratifications together with the maps $\rho^\Phi$ and $\gamma^\Phi$ can be used to compute the resonance. 

\begin{rem} \label{rmk:Strata_P}
The same arguments as in the above proof applied directly to the sheaves $\o_{V(E)}$ and $\mathcal{C}$ define similar stratifications on $\p H^0(X,E)$ and $\gr_2(H^0(X,E))$, respectively, and the restrictions of these strata over $\r(E)$ and $\g(E)$ give the strata of the Theorem \ref{thm:Flattening-Stratification-Square}. Note however, that the strata on $\p H^0(X,E)$ and $\gr_2(H^0(X,E))$ are in general incompatible, i.e. there does not exist a commutative square similar to (\ref{diagrama_rezonanta}), since for a general $\Lambda$, the kernel of the morphism $\Lambda^\vee\to \mathcal{C}|_{\Lambda}$ is of rank two. 
\end{rem}

The strata of Theorem \ref{thm:Flattening-Stratification-Square} are not necessarily irreducible. One example is obtained for a smooth curve $C$ of even genus on a $K3$ surface $S$ with $\Pic(S)$ generated by $\o_S(C)$. It was proved in \cite[Section 4]{aprodu_koszul_2024} that the resonance is a disjoint union of lines, and we observe that the resonance consists of only one non-empty stratum. However, even though the resonance might be reducible, the existence of the flattening stratification leads to the following characterization of its irreducible components.

\begin{prop} \label{prop:IrredCompStrata}
Any irreducible component of $\r(E)$ is an irreducible component of the closure of a stratum.
\end{prop}

\begin{proof}
If $Z$ is an irreducible component of $\r(E)$, then $Z$ has an induced flattening stratification
\[
Z=\textstyle{\bigcup}_\varphi Z^\varphi\mbox{ with } Z^\varphi:=Z\cap\r^\varphi(E).
\]
In this stratification, one stratum $Z^\psi$ is dense,
and hence its closure $\overline{Z^\psi}$ coincides with $Z$. Since $Z$ is an irreducible component of $\r(E)$, the inclusions $\overline{Z^\psi}\subseteq \overline{\r^\psi(E)} \subseteq \r(E)$ show that $Z$ is an irreducible component of $\overline{\r^\psi(E)}$.
\end{proof}

Recall that a projective variety is called \emph{linear} if each of its irreducible components is a projective subspace. An immediate consequence is the following (see \cite[Proposition~6.1]{aprodu_reduced_2024}). 

\begin{prop}
\label{prop:LinearResonance}
The assertions below are equivalent:
    \begin{itemize}
        \item[(a)]
	$\r(E)$ is linear;
	\item[(b)] 
    $\w(E)$ is finite;
	\item[(c)] For every $A \in \Pic(X)$ that can be embedded in $E$ as a saturated sub-line bundle, we have $h^0(E(-A)) = 1$, and there are only finitely many such $A \in \Pic(X)$.
    \end{itemize}
\end{prop}

\begin{proof}
It is clear that (b) and (c) are equivalent.

If $\w(E)$ is finite, then $\r(E)$ is a finite disjoint union of subspaces of type $\p H^0(X,L)$ with $L\in\w(E)$.  Conversely, consider $Z$ an irreducible component of $\r(E)$. By Proposition \ref{prop:IrredCompStrata}, it is an irreducible component of the closure of a stratum $\r^\psi(E)$. In particular, there is a dominant morphism from an open subset $U\subset Z$ to an irreducible component of $\w^{P_X - \psi}(E)$, whose fibers are positive-dimensional projective subspaces. However, since $Z$ is itself a projective space, this is possible only if the morphism in question is constant.
It follows that the corresponding irreducible component of $\w^{P_X - \psi}(E)$ is a point.
\end{proof}

In particular, we see that any linear resonance with two intersecting irreducible components cannot be associated to a vector bundle. This is the case for the resonance of right-angled Artin groups, see \cite[Section 8]{aprodu_reduced_2024} for details. For instance, if
\[
G_\Gamma :=  \langle\ g_1, \ldots, g_n \ : \ g_i \cdot g_{i+1} = g_{i+1} \cdot g_i \ \text{for} \ 1 \le i \le n-1\ \rangle
\]
is the right-angled Artin group associated with the path graph $\Gamma = P_n$, then the projectivized resonance of $G_\Gamma$ --- which is $\r(V,K)$ for a complex vector space $V$ with a basis $e_1, \ldots, e_n$ and $K \subseteq \bigwedge^2 V$ the linear subspace generated by $e_1 \wedge e_2, \ldots, e_{n-1} \wedge e_n$ --- will be a union of $n-2$ hyperplanes in $\p V^\vee$.

\medskip

Now, let us also emphasize the following useful fact.

\begin{prop} \label{prop:dim_g(E)}
If $\dim \g(E) \le 1$, then $h^0(X,L)=2$ for all $L\in \w(E)$ and the map
\[
\gamma:\g(E) \longrightarrow \w (E)
\]
is one-to-one.
\end{prop}

\begin{proof}
From (\ref{diagrama_rezonanta}) and (\ref{diagrama_fiber}), it follows that for any $L \in \w(E)$, we have the inequality 
\[
2 h^0(X,L) - 4 \le \dim \g(E).
\]
Thus, if $\dim \g(E) \le 1$, we obtain that $h^0(X,L) = 2$, and hence the fiber of $\gamma$ over $L$ is a point.
\end{proof}

We highlight an immediate consequence of the previous result in a more specific setting --- namely, when $X$ is a curve and $E$ has rank two --- which will be used later in Section \ref{sec:LowDim}.

\begin{prop}
\label{prop:StabCurve}
Let $E$ be a rank-two vector bundle on a smooth projective curve $X$ such that $\dim \g(E) \le 1$. If either
\begin{itemize}
    \item [(i)] $E$ has canonical determinant and $h^0(X,E) \ge 5$, or
    \item [(ii)] $\det(E) = K_X + D$ with $D\ge 0$ of degree at least 5,
\end{itemize}
then $E$ is stable.
\end{prop}

\begin{proof}
(i) Put $h^0(X,E) = n$. By Proposition \ref{prop:dim_g(E)}, it follows that $h^0(X,L) = 2$ for all saturated subpencils of $E$. In particular, for any exact sequence of type
\[
0\to A\to E\to K_X(-A)\to 0
\]
with $A \in \Pic(X)$ we have $h^0(X,A) \le 2$, and hence $h^0(X,K_X(-A)) \ge n-2$. By Riemann-Roch, we obtain that
\[
\deg(A)\le (g-1)-n+4,
\]
and the latter is strictly less than $\mu(E) = g - 1$, by hypothesis.

(ii) The proof is similar to (i).
\end{proof} 

We will explore the curve case in greater detail in the next section.

\section{Vector bundles on curves}
\label{sec:Curves}

\subsection{Resonance and Brill-Noether loci}
Throughout this section, we assume that $X$ is a smooth projective curve, $E$ is a vector bundle on $X$ of rank at least two. We fix $H$ the ample line bundle associated to a fixed point on $X$. We use the simplified notation $\operatorname{Quot}^{e}_{E^\vee/X}$ for the Quot scheme of rank-one quotients of degree $e$ of $E^\vee$, so that 
\[
\operatorname{Quot}_{E^\vee/X}=\textstyle{\coprod_e} \operatorname{Quot}^{e}_{E^\vee/X}
\]
scheme-theoretically. We also denote by $\w_d(E)$ the locus in $\w(E)$ of sub-line bundles of $E$ of degree $d$, and hence, we obtain a decomposition
\begin{equation}
\label{eqn:DegDecomp}
    \w(E)=\textstyle{\coprod_{d \ge 1}} \w_d(E)
\end{equation}
of algebraic varieties that is connected directly to the decomposition (\ref{eqn:W(E)_strata}) using the identification $L^\vee = L^\sigma$ for any $L\in\w(E)$ (see the second line of the proof of Definition--Lemma \ref{defn-lem_w(E)}, for the definition of $L^\sigma$), which implies that
\[
\w_d(E) = \w^\Phi(E),\mbox{ for the Hilbert polynomial }\Phi(t) = t - (d+g-1).
\]

With the notation from the previous section, if $\r_d(E)$ denotes the inverse image of $\w_d(E)$ via $\rho$ and $\g_d(E)$ denotes the inverse image of $\w_d(E)$ via $\gamma$, i.e.
\[
\r_d(E) = \{[a]\in\r(E):\ \mathrm{deg}(L_a)=d\}\mbox{ and }
\g_d(E) = \{\Lambda\in\g(E):\ \mathrm{deg}(L_\Lambda)=d\},
\]
we immediately notice that
these subsets define the strata in the corresponding flattening stratifications from the proof of Theorem \ref{thm:Flattening-Stratification-Square}. More precisely, we have the identifications
\[
\r_d(E) = \r^\varphi(E)\mbox{ with } \varphi (t) = d\mbox{ for all }t,
\]
and, respectively
\[
\g_d(E) = \g^\psi(E)\mbox{ for the Hilbert polynomial } \psi(t) = t + (d+1-g).
\]
By (\ref{eqn:Strata_Reverse}), the strata corresponding to the maximal value of $d$ such that $\w_d(E)\ne\emptyset$ are automatically closed. Note also that if $E$ is a general vector bundle of rank $r$ and 
\[
\left[\frac{\deg(E)+g-1}{r}\right]\ge 2g,
\]
then, for the maximal value of $d$ such that $\w_d(E)\ne\emptyset$, the locus $\w_d(E)$ coincides with the variety of maximal sub-line bundles $M(E)$ that is related to Segre invariants and higher-rank Brill-Noether theory, see for example \cite{oxbury_2000}. Indeed, in this case, any $L$ in $M(E)$ is of degree at least $g+1$ and hence $h^0(X,L)\ge 2$.

\begin{rem}
\label{rem:DegreeStratification} 
The flattening stratification of the resonance is directly connected to the \emph{degree stratification} introduced in \cite[Section 6]{aprodu_koszul_2024} which has a role in identifying rank-two split bundles on curves. Recall that, for any $k$, we have the following definition
\[
\r_{\ge k}(E):=\{[a]\in \r(E)|\ \mathrm{deg}(L_a)\ge k\}.
\]
Then
\[
\r_{\ge k}(E)=\textstyle{\bigcup}_{d\ge k} \r_{d}(E)
\]
is the flattening stratification of $\r_{\ge k}(E)$, see also \cite[Theorem 1.5]{aprodu_reduced_2024}. 
\end{rem}

In \cite{aprodu_reduced_2024}, the following loci 
\[
W^1_d(E)=\{A\in W^1_d(X) :  h^0(E(-A)) \ne 0 \}
\]
are analyzed in relation with the resonance. They are natural closed subschemes of $W^1_d(X)$, canonically related to $\w(E)$. Using the universal property of the Jacobian and the identification of $\w_d(E)$ with certain sub-line bundles of $E$, there is a natural map
\[
\w_d(E) \longrightarrow W^1_d(E)
\]
whose fiber over a point $A \in W^1_d(E)$ identifies with an open subset in $\p(H^0(E(-A)))$. Note that, since there are only finitely many values of $d$ for which $W^1_d(E)$ is non-empty, the disjoint union (\ref{eqn:DegDecomp}) is finite --- a fact that was already known from the previous section.

As established earlier, the diagram of \emph{sets} (\ref{diagrama_rezonanta_multimi}) is stratified by diagrams of \emph{algebraic varieties} of type

\begin{align} \label{diagrama_rezonanta_d}
\begin{tikzcd}[ampersand replacement = \&, column sep = small]
        \Xi_d(E) \arrow[rr, "p_2"] \arrow[d, swap, "p_1"] \& {}  \& \g_d(E) \arrow[d, "\gamma_d"] \\
        \r_d(E) \arrow[rr, "\rho_d"]  \& {}  \& \w_d(E) \mathrlap{\longrightarrow W^1_d(E)}
\end{tikzcd}
\end{align}
for all $d$, where $\Xi_d(E)=p_2^{-1}(\g_d(E))$.

\medskip

Next, we describe the fibers of the maps $\g_d(E) \longrightarrow W^1_d(E)$ and $\r_d(E) \longrightarrow W^1_d(E)$. 
Fix a point $L \in W^1_d(E)$. Define $U_L \subseteq \p(H^0(E(-L)))$ as the open subset parametrizing sections whose associated morphisms give saturated embeddings of $L$ into $E$.
Note that the set $U_L$ is precisely the fiber over $L$ of the map $\w_d(E) \longrightarrow W^1_d(E)$.
We consider the natural maps 
\[
\theta_L:\p(H^0(L))\times \p(H^0(E(-L))) \longrightarrow \p(H^0(E))
\]
and 
\[
\tau_L:\gr_2(H^0(L))\times \p(H^0(E(-L))) \longrightarrow \gr_2(H^0(E))
\]
induced by multiplication
\[
\mu_{L,E(-L)}:H^0(L)\otimes H^0(E(-L)) \longrightarrow H^0(E).
\]
The map $\tau_L$ is given by
\[
(\Lambda, [\alpha]) \longmapsto \Lambda \cdot \alpha, 
\]
while $\theta_L$ is the projection of the Segre map centered in $\p(\ker(\mu_{L,E(-L)}))$. Put 
\[
\r_L(E):=\theta_L(\p(H^0(L))\times U_L)\mbox{ and }\g_L(E):=\tau_L(\gr_2(H^0(L))\times U_L).
\]
Note that the closure $\overline{\r_L(E)}$ of $\r_L(E)$ coincides with $\theta_L(\p(H^0(L))\times \p(H^0(E(-L))))$ and the closure $\overline{\g_L(E)}$ of $\g_L(E)$ coincides with $\tau_L(\gr_2(H^0(L))\times \p(H^0(E(-L))))$.
Note also that if $h^0(E(-L))=1$, then $\theta_L$ and $\tau_L$ are inclusions and $U_L$ is either a point or the empty set. We readily see that we have the following.

\begin{prop} 
\label{prop:ResSegre}
The fiber over $L$ of the composed morphism 
\(
\r_d(E) \longrightarrow W^1_d(E)
\)
coincides with $\r_L(E)$ and the fiber over $L$ of the morphism
\(
\g_d(E) \longrightarrow W^1_d(E)
\)
is $\g_L(E)$.

More precisely, the fiber over $L$ in $\r_d(E)$ is either:
\begin{itemize}
\item[(i)] the projective space $\p (H^0(L))$ seen as a linear subspace of $\p (H^0(E))$ if $h^0(E(-L))=1$, or
\item[(ii)] the projection in $\p(H^0(E))$ of the open subset in the Segre embedding
\[
\p(H^0(L))\times U_L\subset \p(H^0(L))\times \p(H^0(E(-L))) \hooklongrightarrow \p(H^0(L)\otimes H^0(E(-L)))
\]
via the projectivization of
\(
\mu_{L,E(-L)},
\)
if $h^0(E(-L))\ge 2$.
\end{itemize}

Similarly, the fiber over $L$ in $\g_d(E)$ is either:
\begin{itemize}
\item[(i)] the Grassmannian $\gr_2(H^0(L))$ seen as a sub-Grassmannian of $\gr_2(H^0(E))$ in the case $h^0(E(-L))=1$, or
\item[(ii)] the projection in $\gr_2(H^0(E))$ of the open subset 
\[
\gr_2(H^0(L))\times U_L\subset \gr_2(H^0(L))\times \p(H^0(E(-L)))
\]
via the map $\tau_L$ if $h^0(E(-L))\ge 2$.
\end{itemize}

In particular, if $W^1_d(E)$ is finite, then all the connected components of $\r_d(E)$ and $\g_d(E)$, respectively are of the types described above.
\end{prop}

Clearly, at the points $L\in W^1_d(E)$ where $\theta_L$ is injective (for example if \(\mu_{L,E(-L)}\) is injective), the fiber of the map $\r_d(E)\to W^1_d(E)$ in fact isomorphic to the corresponding open subset in the Segre variety.

Note that, if $W^1_d(E)$ is either empty or consists of a point for any $d\ge 1$, then Proposition \ref{prop:ResSegre} gives a complete description of all strata of the resonance. This is the case for the projective line. A simple example is given by the following (see also Proposition \ref{prop_n = 4, k = 3}).

\begin{exmp}
\label{exmp:O(1)O(1)}
If $X = \p^1$ and $E=\o_{\p^1}(1)\oplus \o_{\p^1}(1)$, then $\r(E) = \r_1(E) = \p^1\times\p^1$. Indeed, $\w(E) = \w_1(E) = \p(H^0(E(-1))$ and the multiplication map $H^0(\o_{\p^1}(1))\otimes H^0(E(-1)) \longrightarrow H^0(E)$ is injective. 
\end{exmp}

In general, for any $L\in W^1_d(E)$, the connection between $\theta_L$ and $\tau_L$ is given by an incidence diagram. These two maps define a morphism 
\[
\xi_L:\Xi_L\times \p(H^0(E(-L))) \longrightarrow \Xi_E.
\]  
If we denote by $\overline{\Xi_L(E)}$ the image of $\xi_L$, then we have the following diagram
\begin{equation}
\begin{tikzcd}[row sep = 1em, column sep = 1em]
\Xi_L \times \p(H^0(E(-L))) \arrow[rr, "p_2 \times \operatorname{id}"] \arrow[dr, "\xi_L"] \arrow[dd, swap, "p_1 \times \operatorname{id}"] &&
  \gr_2(H^0(L)) \times \p(H^0(E(-L))) \arrow[dr, "\tau_L"] \\
& \overline{\Xi_L(E)} \arrow[rr, "\overline{p_2}"] &&
  \overline{\g_L(E)} \arrow[dd] \\
\p(H^0(L)) \times \p(H^0(E(-L))) \arrow[dr, swap, "\theta_L"] \\
& \overline{\r_L(E)} \arrow[rr] \arrow[uu, <-, "\overline{p_1}"] && \{L\}
\end{tikzcd}
\end{equation}

Remark that $\overline{\Xi_L(E)}$ is the incidence variety of $\overline{\r_L(E)}$ and $\overline{\g_L(E)}$. With these preparations, we prove the following.

\begin{prop}
\label{prop:gfinite}
The map $\tau_L$ is generically finite if and only if $\theta_L$ is generically finite.
\end{prop}

\begin{proof}
As we already noticed, the fibers of $\overline{p_1}$ have dimension $h^0(L) - 2$ and the fibers of $\overline{p_2}$ are one-dimensional. Hence, from the surjectivity of the maps $\overline{p_1}$ and $\overline{p_2}$, we obtain the following relation
\begin{equation}
\label{eqn:RL(E)GL(E)}
   \dim \overline{\g_L(E)} - \dim \overline{\r_L(E)} = h^0(L) - 3. 
\end{equation}
To conclude, note that $\theta_L$ is generically finite if and only if $\dim \overline{\r_L(E)} = h^0(L) + h^0(E(-L)) - 2$ and $\tau_L$ is generically finite if and only if $\dim \overline{\g_L(E)} = 2h^0(L) + h^0(E(-L)) - 5$.
\end{proof}

\subsection{Rank-two bundles on the projective line}
We work through some more involved examples on the projective line below. If $L = \cO_{\p^1}(d)$, we shall denote the maps $\theta_L$ and $\tau_L$ by $\theta_d$ and $\tau_d$, respectively. 

\begin{exmp}
\label{exmp:O(1)O(b)}
Choose  $E=\o_{\p^1}(1)\oplus \o_{\p^1}(b)$, with $b \ge 2$.
We observe first that any saturated sub-line bundle of $E$ is isomorphic either to $\o_{\p^1}(1)$ or to $\o_{\p^1}(b)$ and hence the flattening stratification of $\g(E)$ consists of two strata. 
Up to scalar multiplication, there is only one embedding of $\o_{\p^1}(b)$ in $E$, which is saturated. Hence, $\tau_b$ and $\theta_b$ are injective and we get a closed stratum $\g_b = \gr_2(H^0(\o_{\p^1}(b))) \cong \gr_2(b+1)$ of the $\g(E)$ and a closed stratum $\r_b=\p(H^0(\o_{\p^1}(b))) \cong \p^{b}$ of the resonance.

Now, the non-saturated embeddings of $\o_{\p^1}(1)$ are determined by a vector $(0,\alpha)$ where $\alpha\in H^0(\o_{\p^1}(b-1))$, and their saturation is $\o_{\p^1}(b)$. Since $\tau_1$ is again injective, the other stratum $\g_1$ of $\g(E)$ is the complement $U_1$ of the hyperplane $\p(H^0(\o_{\p^1}(b-1)))$ in $\p(H^0(E(-1))) \cong \p^b$.  Since the strata are disjoint, it follows that $\g(E)$ has two irreducible components, namely $\gr_2(H^0(\o_{\p^1}(b)))$ and $\p(H^0(E(-1)))$  intersecting along $\p(H^0(\o_{\p^1}(b-1)))$. The stratum $\r_1$ of $\r(E)$ is the image of $\p(H^0(\cO_{\p^1}(1))) \times U_1$ via the linear projection $\theta_1$. A simple computation shows that $\theta_1(\p(H^0(\cO_{\p^1}(1))) \times \p(H^0(\cO_{\p^1}(b-1)))) = \r_b$ and hence $\r_b \subseteq \overline{\r_1}$. Therefore, $\r(E) = \overline{\r_1}$ is isomorphic to a linear projection of $\p^1 \times \p^b$ from $\p^{2b+1}$ to $\p^{b+2}$.
\end{exmp}

\begin{exmp}
\label{exmp:O(2)O(2)}
Let $E=\o_{\p^1}(2)\oplus \o_{\p^1}(2)$. Any subpencil $L$ of $E$ is isomorphic either to $\o_{\p^1}(1)$ or $\o_{\p^1}(2)$. 
Note that if $L\cong \o_{\p^1}(2)$ then it is automatically saturated, while if $L\cong \o_{\p^1}(1)$ it is either saturated or its saturation is isomorphic to $\o_{\p^1}(2)$. Consequently, 
the flattening stratification gives exactly two strata of $\g(E)$:
\begin{itemize}
    \item[(1)] an open stratum $\g_1$ corresponding to saturated subpencils isomorphic to $\o_{\p^1}(1)$, and 
    \item[(2)] a closed stratum $\g_2$ corresponding to saturated subpencils isomorphic to $\o_{\p^1}(2)$.
\end{itemize}
The stratum $\g_i$ is mapped to $\w_i\subset \w(E)$, for $i\in\{1,2\}$. It is immediate that $\g_1 \longrightarrow \w_1$ is one-to-one and the fibers of $\g_2 \longrightarrow \w_2$ are (dual) projective planes.
By the general properties of the closures of strata in a flattening stratification, we have $\overline{\g_1}\setminus \g_1\subseteq \g_2$. Note that $\tau_1$ and $\tau_2$ are injective. Indeed, if $L\cong\o_{\p^1}(1)$, this is trivial, since $\gr_2( H^0(L))$ is a point. If $L\cong\o_{\p^1}(2)$, assume $\Lambda,\Lambda'\in \gr_2(H^0(L))$, and $(\alpha,\beta), (\alpha,\beta)\in\mathbb{C}^2$ such that $\Lambda\cdot (\alpha,\beta)=\Lambda'\cdot (\alpha',\beta')$. If $\lambda\in\Lambda$ and $\lambda'\in\Lambda'$ such that $\lambda\cdot (\alpha,\beta)=\lambda'\cdot (\alpha',\beta')$ then $\lambda,\lambda'\in H^0(L)$ are proportional. Hence $\lambda'\in \Lambda$ and $\lambda\in\Lambda'$. The description of these two strata is then as follows.
\begin{itemize}
        \item[(1)] $\g_1=\p(H^0(E(-1)))\setminus(\p(H^0(\o_{\p^1}(1))\times \p(H^0(E(-2)))\cong\p^3\setminus (\p^1\times\p^1)$, and 
        \item[(2)] $\g_2= \gr_2(H^0(\o_{\p^1}(2)))\times \p(H^0(E(-2)))\cong\p^2\times \p^1$, and
        \item[(3)] $\overline{\g_1}\cap \g_2=\p^1\times\p^1$.
\end{itemize} 

We pass now to the analysis of the resonance. Denote by $\r_i$ the corresponding strata, and $\Xi_i$ the corresponding incidence loci, for $i\in\{1,2\}$. 
Recall that $\g_1\to \w_1$ is one-to-one, and hence the same is true for the map $\Xi_1\to \r_1$. Since $\Xi_1\to \g_1$ is a $\p^1$-bundle, then, from Proposition \ref{prop:ResSegre}, it follows that $\r_1$ is the image of a linear projection of $\p^1\times (\p^3\setminus (\p^1\times\p^1))$ from $\p^7$ to $\p^5$. In particular, $\overline{\r}_1$ is a hypersurface.  To describe the other stratum $\r_2$, we look at the diagram

\begin{align}
\begin{tikzcd}[ampersand replacement = \&, column sep = small]
        \Xi_2 \arrow[rr, "p_2"] \arrow[d, swap, "p_1"] \& {}  \& \g_2 = \p^2\times\p^1 \arrow[d, "\gamma"] \\
        \r_2 \arrow[rr, "\rho"]  \& {}  \& \w_2(E) = \p^1 \mathrlap{\longrightarrow \{\o_{\p^1}(2)\}}
\end{tikzcd}
\end{align}
and restrict it over $\overline{\g}_1\cap \g_2$ to obtain the diagram
\begin{align}
\begin{tikzcd}[ampersand replacement = \&, column sep = small]
        \Xi_2|_{\overline{\g}_1\cap \g_2} \arrow[rr, "p_2"] \arrow[d, swap, "p_1"] \& {}  \& \overline{\g}_1\cap \g_2 = \p^1\times\p^1 \arrow[d, "\gamma"] \\
        \overline{\r}_1\cap\r_2 \arrow[rr, "\rho"]  \& {}  \& \w_2(E) = \p^1 \mathrlap{\longrightarrow \{\o_{\p^1}(2)\}}
\end{tikzcd}
\end{align}
We observe that the map $\Xi_2|_{\overline{\g}_1\cap \g_2} \to \overline{\r}_1\cap\r_2$ is finite, two-to-one; indeed, for any section $\alpha \in H^0(\o_{\p^1}(1))$, there are exactly two subspaces in $H^0(\o_{\p^1}(2))$ of dimension two that contain $\alpha$, namely the images of $\langle\alpha\rangle\otimes H^0(\o_{\p^1}(1))$ and, respectively $H^0(\o_{\p^1}(1))\otimes \langle\alpha\rangle$ via the multiplication map. The two diagrams above show that both $\r_2$ and $p_1(\Xi_2|_{\overline{\g}_1\cap \g_2})$ are three-dimensional. The successive inclusions
$p_1(\Xi_2|_{\overline{\g}_1\cap \g_2})\subset \overline{\r}_1\cap\r_2\subset \r_2$, together with the irreducibility of $\r_2$ show that $\r_2\subset\overline{\r}_1$.   
The conclusion is that the resonance $\r(E)\subset \p^5$ is a linear projection of $\p^1\times \p^3$ from $\p^7$ to $\p^5$.
\end{exmp}

The previous examples illustrate a pattern that extends to the general case. We now consider the vector bundle $E=\cO_{\p^1}(a) \oplus \cO_{\p^1}(b)$ with $1\le a \le b$ on the projective line. In this setting, any saturated sub-line bundle is isomorphic to $\cO_{\p^1}(d)$ where either $1\le d\le a$ or $d=b$. Accordingly, we obtain a stratum $\r_d(E)$ of the resonance and a stratum $\g_d(E)$ of $\g(E)$. As we have seen, the key is to understand the maps
\[
\theta_d:\p(H^0(\o_{\p^1}(d))) \times \p(H^0(E(-d))) \longrightarrow \p(H^0(E))
\]
and
\[
\tau_d:\gr_2(H^0(\o_{\p^1}(d)))\times \p(H^0(E(-d))) \longrightarrow \gr_2(H^0(E))
\]
whose images $\overline{\r_d(E)}$ and $\overline{\g_d(E)}$, respectively, are the closures of the corresponding strata.

\begin{lem}
\label{lem:tau_d-qfinite}
The morphism $\theta_d$ is finite. In particular, if $1 \le d \le a$, we have
\[
\dim \r_d(E) = a+b-d+1
\mbox{ and }
\dim \g_d(E) = a+b-1,
\]
and, if $a < b$, we have
\[
\dim \r_b(E) = b
\mbox{ and }
\dim \g_b(E) = 2b-2.
\]
\end{lem}

\begin{proof}
First of all, if $a < b$, then $h^0(E(-b)) = 1$ and hence $\theta_b$ is, in fact, injective. Suppose that $1 \le d \le a$. Interpreting $H^0(\o_{\p^1}(m)))$ as $\operatorname{Sym}^m(\c^2)$, the space of homogeneous polynomials of degree $m$ in two variables, the map $\theta_d$ is induced by the multiplication of polynomials:
\[
\theta_d: \p(\operatorname{Sym}^d(\c^2))\times \p(\operatorname{Sym}^{a-d}(\c^2)\oplus \operatorname{Sym}^{b-d}(\c^2)) \longrightarrow \p(\operatorname{Sym}^a(\c^2)\oplus \operatorname{Sym}^b(\c^2))
\]
\[
[(f, g_1 \oplus g_2)] \overset{\theta_d}{\longmapsto} [fg_1 \oplus fg_2].
\]

Since any homogeneous polynomial in two variables is a product of linear forms, and, given a homogeneous polynomial, there are finitely many ways of writing it as a product of homogeneous polynomials of given degrees, the fibers of $\theta_d$ are finite. By Proposition \ref{prop:gfinite}, the map $\tau_d$ is also quasi-finite. By properness, it follows that $\theta_d$ and $\tau_d$ are finite. The computation of the dimensions is straightforward.
\end{proof}

\begin{prop}
\label{prop:R(E)_ProjLine}
We have a chain of inclusions of the closures of the non-empty strata:
\[
\r_b(E) = \overline{\r_b(E)} \subset \overline{\r_a(E)} \subset \overline{\r_{a-1}(E)} \subset \ldots \subset \overline{\r_1(E)}.
\]
In particular, the resonance $\r(E)$ is the the closure of $\r_1(E)$, which is isomorphic to the image of $\p^1 \times \p^{a+b-1} \longrightarrow \p^{a+b+1}$ via the map $\theta_1$.
\end{prop}

\begin{proof}
Let us first note that an element $[h_1 \oplus h_2]\in \p(\operatorname{Sym}^a(\c^2)\oplus \operatorname{Sym}^b(\c^2)) = \p(H^0(E))$ belongs to the resonance if and only if there exists a homogeneous polynomial $f$ of positive degree dividing both $h_1$ and $h_2$.
If $h_1 = 0$ and $h_2 \neq 0$, $[0 \oplus h_2]$ belongs to all non-empty strata $\r_d(E)$. If $h_1 \neq 0$, and $h_2 = 0$ then $[h_1 \oplus 0]$ belongs to $\r_d(E)$ for all $1\le d\le a$. If both $h_1$ and $h_2$ are non-zero, and $r$ is the degree of their greatest common divisor, then $[h_1 \oplus h_2]$ belongs to $\overline{\r_d(E)}$ for all $1\le d\le r$ and is outside $\overline{\r_d(E)}$ if $d > r$.
\end{proof}

\begin{rem}
\label{rem:R(E)_ProjLine}
By the chain of inclusions from Proposition \ref{prop:R(E)_ProjLine} and the general properties of the flattening stratification, we obtain that $\overline{\r_d(E)} = \bigcup_{k \ge d} \r_k(E)$, for $1 \le d \le a$.
\end{rem}

\begin{prop}
\label{prop:G(E)_ProjLine}
The closures of the non-empty strata $\overline{\g_d(E)}$ for $d\in\{1,\ldots, a\} \cup \{ b \}$ are the irreducible components of $\g(E)$.
\end{prop}

\begin{proof}
From Lemma \ref{lem:tau_d-qfinite} it follows that for all $1\le d\le a$, all $\overline{\g_d(E)}$ are irreducible of the same dimension. Since the strata are disjoint, and open in their closures, it follows that $\overline{\g_d(E)}$ are all distinct. In the case where $a < b$, to see that the closed stratum $\g_b(E)$ is not contained in the others, simply note that $\dim \g_b(E) = 2b - 2 \ge a+b-1 = \dim \g_d(E)$.
\end{proof}

We notice an interesting difference between $\r(E)$ and $\g(E)$, beyond the fact that the former is irreducible and the latter is not (except in the case where $a = b = 1$): namely, the stratum $\r_b(E)$ has the smallest dimension among the strata of $\r(E)$, whereas $\g_b(E)$ has the greatest dimension among the strata of $\g(E)$.

\section{Transversal vector bundles}
\label{sec:Transversal}
In the subsequent sections, we address Question \ref{ques:Generic}.

\subsection{General results} 
Let $E$ be a vector bundle of rank at least two on a smooth projective variety $X$ and put $V^\vee = H^0(X,E)$ and $K^\perp = \operatorname{ker}(d_2)$, as usual. 

We begin by recalling the transversality notion we are working with. Let $S_1, S_2$ be smooth subvarieties of some projective space $\p^N$. We say that $S_1$ and $S_2$ intersect \emph{transversely} (or $S_1 \cap S_2$ is a \emph{transversal intersection}) if $T_p S_1 + T_p S_2 = T_p \p^N$, for all $p \in S_1 \cap S_2$. If non-empty, the intersection $S_1 \cap S_2$ will be smooth and $\codim(S_1 \cap S_2) = \codim S_1 + \codim S_2$. The latter condition is sometimes referred to as \emph{dimensional transversality}, see \cite[Chapter 1]{eisenbud_3264_2016} for further details.

Since general linear sections of the Grassmannian are transversal, and we are asking whether these sections arise from vector bundles, it is natural to introduce the following definition.

\begin{defn} \label{def:trans}
The vector bundle $E$ is called \emph{transversal} if $\g$ and $\p K^\perp$ intersect transversely.
\end{defn}

Let us first note that, if the intersection $\g \cap \p K^\perp$ is transversal and positive-dimensional, then it is also irreducible \cite[Theorem 2.1]{fulton_connectivity_1981}. Since $\Xi(E)$ is a $\p^1$-bundle over $\g(E)$, we readily see that both $\Xi(E)$ and $\r(E)$ are irreducible, as well.

\begin{rem} \label{rem:connected_but_not_irred}
In general, if $\g \cap \p K^\perp$ is positive-dimensional, then it is connected \cite[Theorem 2.1]{fulton_connectivity_1981}, but not necessarily irreducible, as the following example shows. Intersecting $\gr_2(\mathbb{C}^4) \hooklongrightarrow \p^5$ with a general $\p^3$ we obtain a smooth quadric $\p^1 \times \p^1$. Intersecting further with a special $\p^2$, we can obtain a union of two lines. 
\end{rem}

\begin{rem}
If $E$ is transversal and $\g(E)$ is zero-dimensional, then $h^0(L) = 2$, for any $L \in \w(E)$ and $\# \w(E) = \# \g(E) = \deg (\g)$. Conversely, if $\g(E)$ is finite, consisting of precisely $\deg (\g)$ points and $\dim \p K^\perp + \dim \g = \dim \p(\bigwedge^2V^\vee)$, then $E$ is transversal.
\end{rem}

For general curves of even genus, the following holds true. 

\begin{prop}
\label{prop:general_curve_Mukai}
Let $C$ be a Brill-Noether-Petri general curve of genus $g = 2k$, with $k\ge 2$, and let $E$ be a globally generated stable rank-two vector bundle on $C$, having canonical determinant and precisely $k+2$ independent global sections. Then $E$ is transversal and $\g(E)$ is finite. More precisely, $\g(E) = W^1_{k+1}(E) = W^1_{k+1}(C)$. 

\end{prop} 

\begin{proof}
A verbatim copy of the proof of \cite[Lemma 4.8]{aprodu_koszul_2024} shows that any $A\in\w(E)$ is a $g^1_{k+1}$. By the definition of the Brill-Nother-Petri generality, the Petri maps 
\[
\mu_A:H^0(C,A)\otimes H^0(C,K_C(-A))\longrightarrow H^0(C,K_C)
\]
are all injective. In particular, by the base-point-free pencil trick, we obtain $h^0(K_C(-2A)) = 0$ for all $A\in W^1_{k+1}(C)$, which implies that $h^0(E(-A)) = 1$ for all $A\in\w(E)$. In particular $\g(E)$ is in bijection to $\w(E)$, by Proposition \ref{prop:dim_g(E)}. From \cite[Proposition 3.1]{mukai_curves_1992}, it follows that any $g^1_{k+1}$ is a subpencil of $E$, and hence $\w(E) = W^1_{k+1}(C)$. Since $C$ is general, $W^1_{k+1}(C)$ consists of 
\(
\frac{(2k)!}{k!(k+1)!}
\) 
distinct points, and hence $\g(E)$ consists of $\deg(\g)$ distinct points. Therefore, $E$ is transversal.

\end{proof}

One typical case when the previous proposition applies is for the restriction of the Lazarsfeld-Mukai bundle over a smooth curve that generates its Picard group, see \cite[Section 4]{aprodu_koszul_2024}.

\medskip

We now state the main result of this section.

\begin{thm} \label{thm:TransvConseq}
If $E$ is a transversal vector bundle, then:
\begin{itemize}
\item[(i)] The map $p_1: \Xi(E) \longrightarrow \p V^\vee$ is smooth at any point $([a], \Lambda)$ such that $L_\Lambda$ is a minimum for the function $\w(E)\ni L\longmapsto h^0(L)\in \z$. In particular, if $L\longmapsto h^0(L)$ is a constant function of value $c$, then $\r(E)$ is smooth of dimension $\dim \g(E) - c + 3$.

\item[(ii)] The map $p_1: \Xi(E) \to \p V^\vee$ is an immersion at any point $([a], \Lambda)$ such that $h^0(L_\Lambda)=2$. In particular, if $p_1$ is an injective map, then it defines an isomorphism between $\Xi(E)$ and $\r(E)$, and hence $\r(E)$ is isomorphic to the $\p^1$-bundle over $\g(E)$ induced by the universal rank-two bundle $\cU$.

\end{itemize}
\end{thm}

\begin{proof}
(i) Let $([a], \Lambda)$ be as in the hypothesis. We will show that the kernel of the differential $dp_1$ has the same dimension in the neighborhood $U = p_2^{-1}(V)$ of $([a], \Lambda)$, where $V$ is formed by those $\Lambda'$ such that $h^0(L_{\Lambda'}) = h^0(L_\Lambda)$. Denote by $\ell_a$ the line determined by $a$ in $V^\vee$. By the transversality hypothesis, the tangent space $T_\Lambda \g(E)$ is just the intersection between $T_\Lambda \g$ and $T_\Lambda \p K^\perp$ inside $T_\Lambda \p (\bigwedge^2 V^\vee)$. The tangent space of the incidence variety $\Xi(E)$ at the point $([a], \Lambda)$ is formed by the pairs of linear maps $(\eta, \varphi)$ making a commutative triangle
\begin{equation*}
\begin{tikzcd}
    \ell_a \arrow[rr, hook] \arrow[dr, swap, "\eta"] && \Lambda \arrow[dl, "\varphi"] \\
    & V^\vee
\end{tikzcd}
\end{equation*}
with $\operatorname{Im}(\varphi) \cap \Lambda = \{0 \}$ such that $\varphi$ is also a tangent vector to $\p K^\perp$ at $\Lambda$, see \cite[p.~205]{harris_first_course_1992}. In coordinates, if $\Lambda = \langle a, b \rangle$, the last condition is equivalent to
\[
a \wedge \varphi(a) + b \wedge \varphi(b) \in K^\perp
\]
The differential $dp_1: T_{([a], \Lambda)} \Xi(E) \to T_{[a]} \p V^\vee$ is given by $(\eta, \varphi) \longmapsto \eta$. If $\eta = 0$, that is $\varphi|_{\ell_a} = 0$, then $b \wedge \varphi(b)\in K^\perp$. By Proposition \ref{prop_gE}, $\varphi(b) \in H^0(X,L_\Lambda)$ modulo $\Lambda$, so the kernel of $dp_1$ at $([a], \Lambda)$ has dimension $c-2$. This property holds true for all the points of $U$, as desired. 

(ii) The first part is derived from (i), while the rest follows from \cite[Theorem 14.9]{harris_first_course_1992}.
\end{proof}

For the projective line, one can classify all the transversal ample bundles of rank two.

\begin{prop}
\label{prop:p1transv}
Let $E=\o_{\p^1}(a)\oplus\o_{\p^1}(b)$ with $a,b\ge 1$ be a rank-two vector bundle on $\p^1$. Then $\gr_2(H^0(E))\cap \p K^\perp$ is transversal if and only if $a=b=1$.
\end{prop}

\begin{proof}
Assume $E = \o_{\p^1}(1)\oplus\o_{\p^1}(1)$, and choose a basis $\{x_0,x_1\}$ for $H^0(\o_{\p^1}(1))$. In this case, the map $\g(E) \longrightarrow \w(E)$ is one-to-one, and we have an identification $\w(E) \cong \p^1$. Any inclusion $\o_{\p^1}(1)\subset E$ is determined by a point $[\alpha:\beta]\in \p^1 \cong \p(H^0(E(-1)))$. 
Then, the corresponding element in $\p^2 \cong \p K^\perp$ is $[\alpha^2:\alpha \beta:\beta^2]$, in other words, $\g(E)\subset \p K^\perp$ is a smooth conic. This implies transversality.

For the converse, we apply Proposition \ref{prop:G(E)_ProjLine} to infer that $\g(E)$ is positive-dimensional and reducible, unless $a = b = 1$.
\end{proof}

\subsection{Transversality in families}
\label{subsec:FlatFamilies}
In this part, we discuss transversality in flat families.
We recall first from \cite[Section 4.2]{aprodu_koszul_2024} the variation of resonance. Assume $X$ is a smooth projective variety, $T$ is a finite-type connected scheme over $\mathbb{C}$ and $\cP$ is a flat family of locally free sheaves on $X\times T$ such that $h^0(X,\cP_t)$ and $h^0(X,\bigwedge^2\cP_t)$ are constant for $t\in T$. Put $\cE:=\pi_{2,*}(\cP)$ and $\cF:=\pi_{2,*}(\cP)$, where $\pi_2:X\times T\to T$ is the projection. By assumption, Grauert's Theorem implies that both $\cE$ and $\cF$ are locally free on $T$. 
Moreover, assume that the natural map
\[
\partial:\textstyle{\bigwedge^2}\cE\to \cF 
\]
is surjective and therefore $\cK^\perp:=\ker(\partial)$ is also locally free. When restricted to a fiber over $t\in T$, the map above coincides with
\[
d_2:\textstyle{\bigwedge^2}H^0(X,\cP_t)\to H^0(X,\textstyle{\bigwedge^2}\cP_t).
\]
Consider $\cG\subset \p(\bigwedge^2\cE)$, the relative Grassmannian over $T$; the fiber over a point $t\in T$ is $\gr_2(H^0(X,\cP_t))$.
The intersection $\cG\cap \p(\cK^\perp)\subset \p(\bigwedge^2\cE)$ admits a natural projection morphism to $T$, whose fiber over a point $t\in T$ is $\gr_2(H^0(X,\cP_t))\cap \p(\cK^\perp)\subset \p(H^0(X,\bigwedge^2\cP_t))$. The incidence diagram (\ref{diagrama_rezonanta}) has a relative version,  and hence we can define the relative resonance loci $\cR\subset \p(\cF)$. With these preparations, we readily prove the following.

\begin{prop}
The locus $\{t\in T|\ \cR_t\ne \emptyset\}$ is closed.
\end{prop}


The most basic example is the following.

\begin{exmp}
If $T=\gr_m(\bigwedge^2V)$ with $m = 2n-3$, then the locus above is the support of the Cayley-Chow form of the Grassmannian $\gr_2(V^\vee)\subset \p(\bigwedge^2V^\vee)$, see \cite{aprodu_koszul_2024} for an extended discussion on this case.
\end{exmp}

A relevant case of resonance variation is the following. Let $X$ be a smooth projective variety and $H$ be an ample line bundle on $X$. Consider $M\subset M_H$ an irreducible component in a moduli space of stable bundles on $X$, and assume that $h^0(E)$ and $\dim(\ker(d_2))$ are constant for bundles $E$ corresponding to points in an open subset $U\subset M$. Then the locus 
\[
\{[E]\in U|\ E \mbox{ is transversal}\}
\]
is open in $U$.
Note that there are cases where $E$ is a special point in the moduli space, where $h^0$ jumps, yet $E$ is transversal. One example is the restricted universal bundle over a one-dimensional linear section of $G(2,6)$, case discussed in detail in Section \ref{sec:LowDim}. However, the transversality of such a bundle does not immediately imply that transversality holds on an open subset in some irreducible component of the moduli space, since the constancy of $h^0$ is required for transversality to propagate.

\section{Universal rank two bundles on linear sections of Grassmannians}
\label{sec:Universal}

As usual, let $V$ be a complex vector space of dimension $n\ge 4$ and let ${\g^\perp} = \gr_2(V)$ be the Grassmannian of two-planes in $V$, embedded in $\p(\bigwedge^2V)$ via the Pl\"ucker embedding. Note that $\dim \g^\perp = 2n-4$. Let also $\cQ$ be the dual of the universal rank-two bundle on $\g^\perp$.

We shall use the following Bott vanishing type theorem on ${\g^\perp}$, see \cite[Corollary 4.1.9]{weyman_cohomology_2003}, \cite[Appendix]{voisin_greens_2002}:

\begin{thm}[Bott vanishing] \label{thm_bott_vanishing}
    For $q\ge 0$ and $j>0$
\[
H^p({\g^\perp}, \operatorname{Sym}^q \cQ(-j)) = 0 \ \text{if}\ p\neq n-2, 2n-4.
\]
Moreover,
\[
H^{n-2}({\g^\perp}, \operatorname{Sym}^q \cQ(-j)) = 0 \ \text{if}\ q + 1 < j.
\]
\end{thm}

With the help of this theorem, we easily infer that:

\begin{cor} \label{cor_bott_vanishing}
If $1\le i \le 2n-5$ then
\begin{itemize}
    \item [(a)] $H^i({\g^\perp}, \o_{{\g^\perp}}(1-i)) = 0$;
    \item [(b)] $H^{i-1}({\g^\perp}, \cQ(-i)) = H^i({\g^\perp}, \cQ(-i)) = 0$;
    \item [(c)] $H^{i-1}({\g^\perp}, \cQ^\vee(1-i)) = 0$.
\end{itemize}
\end{cor}

\begin{proof}
For (a), recall that $H^1({\g^\perp}, \o_{\g^\perp}) = 0$ and also note that, if $n = 4$, Serre duality ensures that $H^2({\g^\perp}, \o_{\g^\perp}(-1)) = H^2({\g^\perp}, \o_{\g^\perp}(-3))$. The rest follows from Theorem \ref{thm_bott_vanishing}. Now, (b) follows directly from Theorem \ref{thm_bott_vanishing} and finally, since $\cQ^\vee$ is a subbundle of $V\otimes \cO_{\g^\perp}$, Theorem \ref{thm_bott_vanishing} ensures also (c). 
\end{proof}

\begin{lem} \label{lem_transversal_vanishing}
If $X$ is a transversal and positive-dimensional linear section of ${\g^\perp}$ and $1\le i \le \dim X - 1$, then
\begin{itemize}
    \item [(a)] $H^i(X, \o_X(1-i)) =  0$;
    \item [(b)] $H^{i-1}(X, \cQ(-i)|_X) = H^i(X, \cQ(-i)|_X) = 0$;
    \item [(c)] $H^{i-1}(X, \cQ^\vee(1-i)|_X) = 0$.
\end{itemize}
\end{lem}

\begin{proof}
We shall proceed by induction on $\codim_{\g^\perp} X$. If $X$ is a hyperplane section of ${\g^\perp}$, twisting the exact sequence
\[
0 \rightarrow \o_{\g^\perp}(-1) \rightarrow \o_{\g^\perp} \rightarrow \o_X \rightarrow 0
\]
by $\o_{\g^\perp}(1-i)$, $\cQ(-i)$ and $\cQ^\vee(1-i)$ respectively, taking cohomology and using Corollary \ref{cor_bott_vanishing}, we obtain the vanishings (a), (b) and (c).

Now, assume that $X\subsetneq Y$ are transversal linear sections of ${\g^\perp}$ with $\dim X = \dim Y - 1 > 0$. Then, the desired vanishings for $X$ follow after twisting the exact sequence
\[
0 \rightarrow \o_Y(-1) \rightarrow \o_Y \rightarrow \o_X \rightarrow 0
\]
by $\o_{\g^\perp}(1-i)$, $\cQ(-i)$ and $\cQ^\vee(1-i)$ respectively, taking cohomology and applying the induction hypothesis for $Y$.
\end{proof}

\begin{thm} \label{thm_E}
Let $K\subset \bigwedge ^2V$ be a subspace such that the intersection $X=\p K\cap {\g^\perp}$ is positive-dimensional and transversal and consider $E=\cQ|_X$. Then $X$ is linearly normal in $\p K$ and we have a natural isomorphism $H^0(X,E)\cong V^\vee$ which identifies $K^\perp$ with $\ker(\operatorname{det})$. Moreover, if $\dim(X)\ge 3$, then $E$ is stable.
\end{thm}

\begin{proof}
The linear normality of $X$ in $\p K$ is equivalent to the surjectivity of the map $\bigwedge^2V^\vee = H^0({\g^\perp}, \o_{\g^\perp}(1))\longrightarrow H^0(X, \o_X(1))$, which is in turn equivalent to
\begin{equation} \label{eq_vanish_1}
    H^1({\g^\perp}, \cI_X(1)) = 0
\end{equation}
Also, to show that $H^0(X,E)$ and $V^\vee$  are naturally isomorphic, it is enough to show that the restriction map $V^\vee \cong H^0({\g^\perp}, \cQ) \longrightarrow H^0(X, E)$ is an isomorphism. For this to happen, it suffices to prove that
\begin{equation} \label{eq_vanish_2}
    H^0({\g^\perp}, \cI_X\otimes \cQ) = H^1({\g^\perp}, \cI_X\otimes \cQ) = 0
\end{equation}
We will prove (\ref{eq_vanish_1}) and (\ref{eq_vanish_2}) by induction on $\codim_{\g^\perp} X$. If $X$ is a hyperplane section of ${\g^\perp}$, then $\cI_X = \o_{\g^\perp}(-1)$, and hence the conclusion follows using Corollary \ref{cor_bott_vanishing}.

Suppose now that $X\subsetneq Y$ are transversal linear sections of ${\g^\perp}$ with $\dim X = \dim Y - 1 > 0$. The desired vanishings easily follow after twisting the exact sequence
\[
0 \rightarrow \cI_Y \rightarrow \cI_X \rightarrow \o_Y(-1) \rightarrow 0
\]
by $\o_{\g^\perp}(1)$ and $\cQ$ respectively, taking cohomology and applying the induction hypothesis and Lemma \ref{lem_transversal_vanishing} with $i = 1$.

Consequently, the surjection $\bigwedge^2V^\vee \twoheadlongrightarrow K^\vee$ can be obtained by taking the second exterior power of the natural map $V^\vee \otimes \cO_X \longrightarrow E$ and then passing to global sections, which is in turn the determinant map of $E$.

For the last claim, we use the fact that the Picard group of the Grassmannian is generated by the hyperplane section bundle $\o_{\g^\perp}(1)$ and apply inductively the Lefschetz Theorem for Picard groups (see \cite[Example 3.1.25]{lazarsfeld_positivity_2004}) to infer that $\Pic(X)$ is generated by the hyperplane section bundle as well. Let $\cL\subset E^\vee$ be a rank-one subsheaf. Without loss of generality, we may assume $\cL$ is locally free, and hence $\cL=\cO_X(a)$ for some integer $a$. Due to Lemma \ref{lem_transversal_vanishing}, $E^\vee$ has no non-zero global sections, which implies that $a<0$. Therefore, since $\deg E = \deg {\g^\perp}$
\[
\mu(\cL)=a\cdot \deg {\g^\perp} < \frac{-\deg {\g^\perp}}{2} = \mu(E^\vee)
\]
proving the stability of $E^\vee$.
\end{proof}

\begin{rem}
By construction, $E$ is globally generated and its determinant map is surjective. If $\dim X  \le 2$, then $E$ is not necessarily stable, as Propositions \ref{prop_n = 4, k = 3} and \ref{prop_n = 4, k = 4} show.
\end{rem}

\begin{rem} \label{rem:delPezzo5}
By our construction, $X$ is not an outer projection of a variety in a larger projective space, in particular, for $n = 5$ and $\dim K = 6$, $X$ will be a smooth del Pezzo surface, cf. \cite[Definition 8.1.5]{dolgachev_2012}.
\end{rem}

An immediate consequence of Theorem \ref{thm_E} is the following positive answer to Question \ref{ques:General}.

\begin{cor}
\label{cor_E}
Let $K\subset \bigwedge^2V$. If $X = \p K \cap {\g^\perp}$ is transversal and positive-dimensional, then $\mathbb{R}(V,K) = \mathbb{R}(E)$, where $E = \cQ|_X$.
\end{cor}

Note that the hypotheses of Corollary \ref{cor_E} do not ensure the transversality of the intersection $\p K^\perp \cap {\g}$. In other words, it is unclear whether the bundle $E = \cQ|_X$ is transversal. We will address this issue next, for small values of $n$.

\section{Low-dimensional Grassmannians and resonance}
\label{sec:LowDim}

The goal of this section is to illustrate Theorem \ref{thm_E} for small values of $n = \dim V$, by describing $X\hooklongrightarrow \p K$, the rank-two bundle $E$ on $X$ and the resonance variety $\r(E) = \r(V,K)$.

We recall that for $n \in \{4, 5\}$, the Grassmannian $\g^\perp = \gr_2( V)$ is self-dual, in the sense that its projective dual is $\g = \gr_2( V^\vee)$. As noticed for instance in \cite[Proposition 2.24]{debarre_gushel-mukai_2018}, this self-duality property ensures the following.

\begin{thm} \label{thm:LowDim_transversal_n=4,5}
If $n\in\{4,5\}$ and $ \p K \cap {\g^\perp}$ and $\p K^\perp\cap \g$ are both non-empty, then 
$\p K \cap {\g^\perp}$ is transversal if and only if $\p K^\perp\cap \g$ is transversal. 
\end{thm}

We extend this theorem for $n = 6$ in Theorem \ref{thm:g(8,15)}. In this case, $\g$ is no longer self-dual and what we use instead is a reinterpretation of the results in \cite{mukai_curves_1992}. Note that if $n\ge 7$ and $\p K \cap {\g^\perp}$ is transversal and non-empty, then  $\p K^\perp\cap \g$ is transversal if and only if it is empty. This observation demonstrates the relevance of the cases $n\le 6$.

In the process, we also observe that the condition requiring the intersection to have dimension at least three in Theorem \ref{thm_E} is necessary to ensure the stability of $E = \cQ|_X$.

\subsection{Dimension 4 case}

In the sequel, $\dim V = 4$. Then ${\g^\perp} = \gr_2(V)$ will be a smooth hyperquadric in $\p^5$.

\begin{prop} \label{prop_n = 4, k = 3}
Let $K\subset \bigwedge^2V$ of dimension $3$ such that $X = \p K \cap {\g^\perp}$ is transversal. Then $X\cong \p^1$ is a smooth conic in $\p^2$, $E = \cO_{\p^1}(1) \oplus \cO_{\p^1}(1)$ and $\r(E) = \r(V,K)$ is isomorphic to the trivially ruled surface $\p^1 \times \p^1$.
\end{prop}

\begin{proof}
Since $\g^\perp$ has degree $2$, we may identify $X$ with the image of the Veronese embedding $\nu_2 : \p^1 \hooklongrightarrow \p^2$. By Theorem \ref{thm:LowDim_transversal_n=4,5}, we infer that $\g(E) \cong \nu_2(\p^1)$ as well. Therefore, with Proposition \ref{prop:dim_g(E)} in mind, we deduce that $h^0(X,L) = 2$ for any $L \in \w(E)$. Since $c_1(E) = \cO_{\p^1}(2)$ and $E$ is globally generated, it follows that $E = \cO_{\p^1}(1) \oplus \cO_{\p^1}(1)$. Now, Theorem \ref{thm:TransvConseq} implies that the resonance $\r(E)$ is the projectivization of the universal rank-two bundle on $\g(E)$, which is $\p(\cO_{\p^1}(-1)\oplus \cO_{\p^1}(-1))\cong \p^1 \times \p^1$. 
\end{proof}

\begin{prop} \label{prop_n = 4, k = 4}
Let $K\subset \bigwedge^2V$ of dimension $4$ such that $X = \p K \cap {\g^\perp}$ is transversal. Then $X \cong \p^1 \times \p^1$ is a smooth quadric in $\p^3$,  $E = \cO_X(1,0) \oplus \cO_X(0,1)$ and $\r(E)$ is a disjoint union of $2$ lines in $\p^3$.
\end{prop}

\begin{proof}
Again, since $\deg {\g^\perp} = 2$, we can identify $X$ with the image of the Segre embedding $\sigma_{1,1} : \p^1\times \p^1 \hooklongrightarrow \p^3$. Note that $c_1(E) = \cO_X(1,1)$ and $c_2(E) = 1$, by standard Schubert calculus, see \cite[Section 5.6]{eisenbud_3264_2016}. It was shown in \cite[Lemma 1]{brinzanescu_algebraic_1991} that $E$ lies in an extension of the form
\[
0 \longrightarrow \cO_X(\alpha,\beta) \longrightarrow E \longrightarrow \cO_X(1-\alpha, 1-\beta) \otimes \cI_Z \longrightarrow 0
\]
where $Z$ is a zero-dimensional scheme of length $l(Z) = 1- \alpha (1 - \beta) - \beta (1 - \alpha)$ and $\alpha \ge 1-\alpha$. But since $E$ is globally generated, we immediately see that $\alpha = 1$ and then $\beta = l(Z) = 0$. Since $h^1(\cO_X(1,-1)) = 0$, the bundle $E$ is nothing but $\cO_X(1,0) \oplus \cO_X(0,1)$. Finally, by Theorem \ref{thm:LowDim_transversal_n=4,5}, $\g(E)$ consists of $2$ points and therefore, $\r(E)$ will be a disjoint union of $2$ lines in $\p^3$.
\end{proof}

\begin{rem}
Any two disjoint lines in $\p V^\vee \cong \p^3$ appear as a resonance variety. Indeed, if $\ell_{ab}$ is the line joining the points $[a]$ and $[b]$ and $\ell_{cd}$ is the line joining $[c]$ and $[d]$, then $\ell_{ab}$ is disjoint from $\ell_{cd}$  if and only if $\{a, b, c, d\}$ is a basis of $\p V^\vee$ if and only if the line passing through $[a \wedge b]$ and $[c \wedge d]$ intersects the Grassmannian $\gr_2( V^\vee)$ only in these two points. Therefore, if the subspace $K^\perp$ is generated by $a \wedge b$ and $c \wedge d$, then $\r(V,K) = \ell_{ab} \sqcup \ell_{cd}$. Moreover, the intersection $X = {\g^\perp} \cap \p K$ will be transversal and hence $\r(E) = \r(\cQ|_X) = \ell_{ab} \sqcup \ell_{cd}$.
\end{rem}

\subsection{Dimension 5 case}

In the sequel, $n = \dim V = 5$. Then ${\g^\perp} = \gr_2(V)$ will be a $6$-dimensional smooth projective variety of degree $5$ in $\p^9$, whose canonical divisor in $\cO_{\g^\perp}(-5)$.

\begin{prop} \label{prop_n = 5, k = 5}
Let $K\subset \bigwedge^2V$ of dimension $5$ such that $X = \p K \cap {\g^\perp}$ is transversal. Then $X$ is an elliptic normal quintic curve in $\p^4$, $E$ is the unique stable rank-two bundle with the determinant $\cO_X(1)$ and $\r(V,K)$ is isomorphic to the ruled surface $\p(E)$.
\end{prop}

\begin{proof}
By the adjunction formula and the fact that $\g^\perp$ has degree $5$, it follows that $X$ is an elliptic normal quintic in $\p K$. By Theorem \ref{thm:LowDim_transversal_n=4,5}, $\g(E)$ will be also an elliptic normal quintic in $\p K^\perp$ and thus, using Proposition \ref{prop:dim_g(E)}, we have that $h^0(L) = 2$, for any $L \in \w(E)$. By Riemann-Roch, $\deg(L) = 2$, for any $L\in \w(E)$, in other words $W_2(E)= \w(E)$, with the notation of Section \ref{sec:Curves}. It follows from Proposition \ref{prop:StabCurve} that $E$ is stable. Using Atiyah's classification of indecomposable rank-two bundles on $X$, we infer that $E$ is the unique rank-two bundle having the determinant $\cO_X(1)$. Also, by Theorem \ref{thm:TransvConseq}, the resonance $\r(E)$ is the projectivization of the universal rank-two bundle on $\g(E)$. But $\g(E)$ is in fact isomorphic to $X$. Indeed, since $h^0(X,E) = 5$, we have that $h^0(E(-L)) \neq 0$, for any $L \in \Pic_2(X)$ and therefore $X \cong \Pic_2(X) = \g_2(E) = \g(E)$. Consequently, $\r(E)$ is isomorphic to the ruled surface $\p(E^\vee) = \p(E)$.
\end{proof}

\begin{rem}
Any elliptic normal quintic can be realized as a transversal linear section of $\gr_2(\mathbb{C}^5)$, as explained, for example, in \cite{fisher_invariant_2013}.
\end{rem}

\begin{prop} \label{prop_n = 5, k = 6_part_1}
Let $K\subset \bigwedge^2V$ of dimension $6$ such that $X = \p K \cap {\g^\perp}$ is transversal. Then $X$ is a smooth quintic del Pezzo surface in $\p^5$ and $\r(E) = \r(V,K)$ a disjoint union of $5$ lines in $\p^4$.
\end{prop}

\begin{proof}
By the adjunction formula, $K_X = \cO_X(-1)$. Since the degree of $\g^\perp$ is $5$, the conclusion follows. Applying Theorem \ref{thm:LowDim_transversal_n=4,5}, we infer that $\g(E)$ consists of $5$ distinct points. Thus, by Proposition \ref{prop:dim_g(E)}, the map $\g(E) \longrightarrow \w(E)$ is one-to-one and hence, by Theorem \ref{thm:TransvConseq}, the resonance $\r(E)$ is a disjoint union of $5$ lines in $\p^4$.
\end{proof}

Recall that any quintic del Pezzo surface in $\p^5$ is a linear section of $\gr_2(\mathbb{C}^5)$, see, for example \cite{fujita_1981}, \cite[Proposition 8.5.1]{dolgachev_2012}.
Projective self-duality of $\gr_2(\mathbb{C}^5)$ and Proposition \ref{prop_n = 5, k = 6_part_1} lead to the following answer to Questions \ref{ques:General} and \ref{ques:Generic} in the particular case $n=5$.

\begin{cor}
    Any resonance variety consisting of five disjoint lines in $\p^4$ is the resonance of a rank-two vector bundle on the smooth quintic del Pezzo surface.
\end{cor}

Let us now give a description of the bundle $E$ in the context of Proposition \ref{prop_n = 5, k = 6_part_1}. It is well-known (see, for example, \cite[Chapter 8]{dolgachev_2012}) that a smooth del Pezzo surface $X$ of degree $5$ in $\p^5$ is a blow-up of $\p^2$ in $4$ general points $X \xlongrightarrow{\sigma} \p^2$ and it is embedded in $\p^5$ by the system of cubics passing through these points. If $E_1, \ldots, E_4$ are the exceptional curves, then $\sigma^*\cO_{\p^2}(2) \otimes \cO_X(-E_1 - E_2 - E_3 - E_4)$ and $\sigma^*\cO_{\p^2}(1) \otimes \cO_X(-E_i)$, for $i = \overline{1,4}$, are the only base point free pencils on $X$ (see, for example \cite[Section 2]{dolgachev_geometry_2018}). 

\begin{prop} \label{prop_n = 5, k = 6_part_2}
Let $K\subset \bigwedge^2V$ of dimension $6$ such that $X = \p K \cap {\g^\perp}$ is transversal. Then $X \xlongrightarrow{\sigma} \p^2$ is a blow-up of $\p^2$ in $4$ general points and $E$ is the unique nontrivial extension
\begin{equation} \label{eq:E_delPezzo}
0 \longrightarrow \sigma^*\cO_{\p^2}(2) \otimes \cO_X(-E_1 - E_2 - E_3 - E_4) \longrightarrow E \longrightarrow \sigma^*\cO_{\p^2}(1) \longrightarrow 0
\end{equation}
\end{prop}

\begin{proof}
By the discussion above, $c_1(E) = \sigma^*\cO_{\p^2}(2) \otimes \cO_X(-E_1-E_2-E_3-E_4)$. Also, $c_2(E) = 2$, by standard Schubert calculus, see \cite[Section 5.6]{eisenbud_3264_2016}. Then, it readily follows that $E$ lies in an extension (\ref{eq:E_delPezzo}). Since $h^0(X,E) = 5$ and $h^0(X,L) = 2$, for any $L \in \w(E)$, it follows that $E$ is not split. Now, by a simple computation, $h^1(\sigma^*\cO_{\p^2}(1) \otimes \cO_X(-E_1 - E_2 - E_3 - E_4)) = 1$, showing the uniqueness of the extension (\ref{eq:E_delPezzo}). We also note that $\w(E)$ is formed by the five pencils on $X$.
\end{proof}

\subsection{Dimension 6 case}

In the sequel, $n = \dim V = 6$. Then $\g^\perp = \gr_2(V)$ will be an $8$-dimensional smooth projective variety of degree $14$ in $\p^{14}$. Its projective dual is a cubic hypersurface in $\p(\bigwedge^2 V^\vee)$, whose singular locus is $\g = \gr_2( V^\vee)$.  We summarize the main results of \cite{mukai_curves_1992} in the following theorem.

\begin{thm}[Mukai] \label{thm_mukai}
(a) A one-dimensional transversal linear section of ${\g^\perp}$ is a canonical curve $C$ of genus $8$ without $g^2_7$. \label{thm_a} \\
(b) Any canonical curve $C$ of genus $8$ without $g^2_7$ can be obtained as a transversal linear section of $\g^\perp \subset \p^{14}$, as follows. Up to an isomorphism, there is a unique stable rank-two bundle $E_M$ on $C$ with canonical determinant and precisely $6$ independent global sections. Moreover, the determinant map of $E_M$ is surjective. Setting $V:= H^0(E_M)^\vee$ and $K_M:= H^0(\omega_C)^\vee \hooklongrightarrow \bigwedge^2 V$ there is an isomorphism $C \cong \p K_M \cap \g^\perp$.  \\
(c) Furthermore, if $L$ is a $g^1_5$ of $C$ and $M = K_C \otimes L^\vee$ is the Serre adjoint of $L$, then $E_M$ is the unique nontrivial extension
\[
0 \longrightarrow L \longrightarrow E_M \longrightarrow M \longrightarrow 0 \ .
\]
(d) If $\det$ denotes the determinant map of $E_M$, then there is a bijection between the intersection $\p \operatorname{Ker}(\det) \cap \g = \p K_M^\perp \cap \g$ and the finite set $W^1_5(C)$ of $g^1_5$'s on the curve $C$.
\end{thm}

We further remark that a canonical curve of $C \subset \p^7$ of genus $8$ is non-degenerate, so, if it has no $g^2_7$, there exists a unique $7$-dimensional linear space in $\p^{14}$ cutting out the Grassmannian $G^\perp$ along $C$.  Also, as noticed for instance in \cite{mukai_2022}, such a curve $C$ is Brill-Noether-Petri general and hence, the set $W^1_5(C)$ is smooth and consists of $14$ points, see \cite{Griffiths-Harris_1980}. 

We connect Mukai's results with our approach as follows.

\begin{prop}
\label{prop:E=E_M}
Let $K\subset \bigwedge^2V$ of dimension $8$ such that $X = \p K \cap {\g^\perp}$ is transversal. Then $X$ is a canonical curve of genus $8$ without $g^2_7$, $E$ is the rank-two bundle $E_M$ defined in Theorem \ref{thm_mukai} and the resonance $\r(E)$ is a disjoint union of $14$ lines in $\p^5$.
\end{prop}

\begin{proof}
By Theorem \ref{thm_mukai}, $X$ will be a canonical curve in $\p^7$ without $g^2_7$. Then, by its definition, we see that $E$ has canonical determinant and $h^0(X,E) = 6$. To conclude, it remains to show that $E$ is stable. By the discussion above, $K = K_M$ and hence the resonance of $E$, which is $\r(V,K)$ coincides with the resonance of $E_M$, which is a disjoint union of $14$ lines. Therefore, $\g(E)$ consists of $14$ points, ensuring that $h^0(L) = 2$, for all $L \in \w(E)$. Then, by Proposition \ref{prop:StabCurve}, $E$ is stable, and hence $E = E_M$.
\end{proof}

Mukai's Theorem \ref{thm_mukai} implies tranversality in one direction, i.e. if $\p K \cap \g^\perp$ is transversal, then $\p K^\perp \cap \g$ is transversal. Next we focus on showing transversality in the opposite direction. In order to do so, let us introduce the following subvariety:
\[
\cD = \left\{K\in \gr_8(\textstyle{\bigwedge^2V)}: \p K \ \text{and}\ \g^\perp \ \text{do not intersect transversely} \right\}
\]

\begin{lem} \label{lem_ired_d}
$\cD$ is an irreducible subvariety of $\gr_8(\bigwedge^2V)$.
\end{lem}

\begin{proof}
Let $\Xi$ be the incidence variety of $\g^\perp$ and $\gr_8(\bigwedge^2V)$ and let $\Sigma\subset \Xi$ be the locus given by the pairs $(q,K)$ such that $\p K$ and $\g^\perp$ are not transverse at $q$. Note that $\Sigma$ inherits from $\Xi$ a map $\sigma_1$ to $\g^\perp$, which is surjective and a map $\sigma_2$ to $\gr_8(\bigwedge^2V)$, whose image is precisely $\cD$.

Let us compute the fibers of $\sigma_1$. If $q \in \g^\perp$ is the point representing a line $\ell_q \subset \bigwedge^2V$, then $\sigma_1^{-1}(q)$ is formed by those $K$ such that $\ell_q \subset K$ and  $\dim(K/\ell_q \cap T_q \g^\perp) \ge 2$, which can be identified with the Schubert cycle $\Sigma_{1,1}(\cV)\subset \gr_7(\bigwedge^2V/\ell_q)$, where $\cV$ is a complete flag
\[
0 \subset V_1 \subset \ldots \subset V_{14} = \textstyle{\bigwedge^2V/\ell_q}
\]
such that $V_8 = T_q\g^\perp$. Consequently, $\sigma_1$ has irreducible and equidimensional fibers, see \cite[Theorem 4.1]{eisenbud_3264_2016}. Since $\g^\perp$ is also irreducible, so are $\Sigma$ and $\cD$.
\end{proof}

Let us now define a second subvariety of $\gr_8(\bigwedge^2V)$:

\[
\cD^\perp = \left\{K\in \gr_8(\textstyle{\bigwedge^2V)}: \p K^\perp \ \text{and}\ \g \ \text{do not intersect transversely} \right\}
\]

\begin{lem} \label{lem_ired_d'}
$\cD^\perp$ is an irreducible divisor of $\gr_8(\bigwedge^2V)$.
\end{lem}

\begin{proof}
Let $\Xi^\perp$ be the incidence variety of $\g$ and $\gr_7(\bigwedge^2V^\vee)$ and let $\Sigma^\perp\subset \Xi^\perp$ given by the pairs $(r, K^\perp)$ such that $\p K^\perp$ and $\g$ are not transverse at $r$. Let $\sigma_1^\perp$ and $\sigma_2^\perp$ be the maps to $\g$ and $\gr_7(\bigwedge^2V^\vee)$, respectively. Via the isomorphism $\gr_8(\bigwedge^2V) \xlongrightarrow{\sim} \gr_7(\bigwedge^2V^\vee)$, $K \longmapsto K^\perp$, it suffices to show that the image of $\sigma_2^\perp$ is an irreducible divisor.

Let us compute the fibers of $\sigma_1^\perp$. If $r \in \g$ in the point representing a line $\ell_{r} \subset \bigwedge^2V^\vee$, then the fiber of $r$ via $\sigma_1^\perp$ is formed by those $K^\perp$ such that $\ell_{r} \subset K^\perp$ and  $K^\perp/\ell_{r} \cap T_{r}\g \neq \{0\}$, which can be identified with the Schubert cycle $\Sigma_1(\cV')\subset \gr_6(\bigwedge^2V^\vee/\ell_{r})$, where $\cV'$ is a complete flag
\[
0 \subset V_1' \subset \ldots \subset V_{14}' = \textstyle{\bigwedge^2V^\vee/\ell_{r}}
\]
such that $V_8' = T_{r}\g$. Therefore, the fibers of $\sigma_1^\perp$ are irreducible divisors of $\gr_6(\mathbb{C}^{14})$, see \cite[Theorem 4.1]{eisenbud_3264_2016}. Thus, $\dim \Xi^\perp = 55$. Then note that $\sigma_2^\perp$ are has generically finite fibers, so $\cD^\perp$ is irreducible of dimension $55 = \dim \gr_8( \bigwedge^2V) - 1$.
\end{proof}

As mentioned earlier, Theorem \ref{thm_mukai} shows that $\operatorname{supp}(\cD^\perp)\subseteq \operatorname{supp}(\cD)$. Taking Lemmas \ref{lem_ired_d} and \ref{lem_ired_d'} into account, we deduce the reverse inclusion.

\begin{thm}
\label{thm:g(8,15)}
Let $K\in \gr_8( \bigwedge^2V)$. The intersection $\p K \cap \gr_2(V)$ is transversal if and only if the intersection $\p K^\perp \cap \gr_2( V^\vee)$ is transversal.
\end{thm}

An immediate consequence is the following positive answer to 
Question \ref{ques:Generic} for $n = 6$.

\begin{cor}
Any resonance variety consisting of fourteen disjoint lines in $\p^5$ is the resonance of a Mukai bundle on a general curve of genus 8.
\end{cor}

\printbibliography

\end{document}